\documentclass[a4paper,11pt]{article}
\usepackage[utf8]{inputenc}
\usepackage[english]{babel}

\usepackage{amsmath}
\usepackage{amsfonts}
\usepackage{amssymb}
\usepackage{amsthm}
\usepackage{mathrsfs}
\usepackage{cite}
\usepackage{calrsfs,authblk}
\allowdisplaybreaks

\usepackage[dvipsnames]{xcolor}
\usepackage{tikz}
\usetikzlibrary{arrows.meta}
\usetikzlibrary{calc}

\pgfmathsetseed{3}

\usepackage{enumitem}
\usepackage{graphicx}
\usepackage{hyperref}
\usepackage{tikz}
\usepackage{array}
\usepackage{multirow}
\usepackage{xcolor}
\usepackage{float}
\usepackage{caption}
\usepackage{subcaption}
\usepackage{titlesec}
\usepackage[Conny]{fncychap}
\usepackage[T1]{fontenc}
\usepackage{fancyhdr}
\usepackage{geometry}
\usepackage{eso-pic}
\usepackage{tabto}

\renewcommand{\geq}{\geqslant}
\renewcommand{\leq}{\leqslant}
\renewcommand{\epsilon}{\varepsilon}

\let\OLDthebibliography\thebibliography
\renewcommand\thebibliography[1]{
  \OLDthebibliography{#1}
  \setlength{\parskip}{0pt}
  \setlength{\itemsep}{3pt plus 0.3ex}
}

\hypersetup{colorlinks=true,    linkcolor=MyDarkBlue,
	citecolor=BrickRed,       filecolor=BrickRed,       urlcolor=darkgreen
}
\usepackage{color}
\definecolor{darkgreen}{rgb}{0,0.4,0}
\definecolor{MyDarkBlue}{rgb}{0,0.08,0.85}
\definecolor{BrickRed}{rgb}{0.8,0.08,0}

\newtheorem{theorem}{Theorem}
\newtheorem{lemma}[theorem]{Lemma}

\newtheorem{proposition}[theorem]{Proposition}

\author[1]{Viet Hung Hoang}
\author[2,3]{Kilian Raschel}
\affil[1]{Faculty of Fundamental Science, Industrial University of Ho Chi Minh City, Vietnam, \href{mailto:hoangviethung@iuh.edu.vn}{\texttt{hoangviethung@iuh.edu.vn}}}
\affil[2]{Université d'Angers, CNRS, France, \href{mailto:raschel@math.cnrs.fr}{\texttt{raschel@math.cnrs.fr}}}
\affil[3]{International Research Laboratory France-Vietnam in Mathematics and its Applications, CNRS - VAST - VIASM}
\title{\vspace{-1cm}\textbf{Boundary contacts for reflected random walks in the quarter plane}}

\begin{document}
\maketitle

\begin{abstract}
We investigate reflected random walks in the quarter plane, with particular emphasis on the time spent along the reflection boundary axes. Assuming the drift of the random walk lies within the cone, the local time converges — without the need for normalization — to a limiting random variable as the walk length tends to infinity. This paper focuses on the properties of these discrete limiting variables. The problem is rooted in probability theory but also has natural connections to statistical physics and analytic combinatorics. We present two main sets of results, each based on different assumptions regarding the random walk parameters. First, when the reflections on the horizontal and vertical boundaries are assumed to be similar, we reveal the recursive structure of the problem through a coupling approach. Second, in the case of more general reflection rules but singular random walks, we derive an explicit closed-form expression for the limiting distribution using the compensation approach. Our results are illustrated via concrete computations on various examples.
\end{abstract}

\section{Introduction and main results}

In this paper, we study a class of reflected random walks confined to cones, with particular attention to their local times on the boundary. Such problems have a rich history in probability theory and are closely linked to various other areas, which we briefly outline below.

To present the problem, consider a simple random walk on the integers, with jumps of $+1$ and $-1$, as shown on the left of Figure~\ref{fig:RW_dim1}. In this context, we introduce the local time $Z_0(n)$ at $0$, which is the amount of time that the random walk spends at $0$ (starting at $0$, say) before time $n\geq0$. If the random walk is balanced, i.e., if the probabilities of moving up and down are both $\frac{1}{2}$, then it is well known that the random variable $\frac{Z_0(n)}{\sqrt{n}}$ converges in distribution to a half-normal distribution. On the other hand, if the random walk has a non-zero drift, then $Z_0(n)$ converges to a geometric random variable without normalisation. These results admit various broad generalisations in higher-dimensional lattices and for more general random walk distributions. Next, it is natural to consider inhomogeneous models of random walks, such as reflected random walks; see the right side of Figure~\ref{fig:RW_dim1}. A key issue is understanding how reflection at $0$ affects the distribution of local time. In some simple examples, if the random walk has a positive drift, the distribution may reduce to a geometric one.

\begin{figure}
\begin{center}
  \begin{tikzpicture}[
    x=5.5mm,
    y=5.5mm,
  ]

    \draw[thin]
      \foreach \x in {0, ..., 10} {
        (\x, -3) -- (\x, 6)
      }
      \foreach \y in {-2, ..., 5} {
        (0, \y) -- (11, \y)
      }
    ;  

    \begin{scope}[
      semithick,
      ->,
      >={Stealth[]},
    ]
      \draw (0, -3.5) -- (0, 6.5);
      \draw (0, 0) -- (11.5, 0);
    \end{scope}
  \draw[thick] (0,0) -- (1,1) -- (2,2) -- (3,1) -- (4,0) -- (5,-1) -- (6,0) -- (7,-1) -- (8,-2) -- (9,-1) -- (10,0) ;
\filldraw[red] (0,0) circle (2.5pt);
\filldraw[blue] (1,1) circle (2.5pt);
\filldraw[blue] (2,2) circle (2.5pt);
\filldraw[blue] (3,1) circle (2.5pt);
\filldraw[red] (4,0) circle (2.5pt);
\filldraw[blue] (5,-1) circle (2.5pt);
\filldraw[red] (6,0) circle (2.5pt);
\filldraw[blue] (7,-1) circle (2.5pt);
\filldraw[blue] (8,-2) circle (2.5pt);
\filldraw[blue] (9,-1) circle (2.5pt);
\filldraw[red] (10,0) circle (2.5pt);
    \end{tikzpicture}\qquad
  \begin{tikzpicture}[
    x=5.5mm,
    y=5.5mm,
  ]

     \draw [fill=white,gray!20,draw opacity=1] (0,0) rectangle (11,-3);
    \draw[thin]
      \foreach \x in {0, ..., 10} {
        (\x, -3) -- (\x, 6)
      }
      \foreach \y in {-2, ..., 5} {
        (0, \y) -- (11, \y)
      }
    ;  

    \begin{scope}[
      semithick,
      ->,
      >={Stealth[]},
    ]
   
      \draw (0, -3.5) -- (0, 6.5);
      \draw (0, 0) -- (11.5, 0);
    \end{scope}
  \draw[thick] (0,0) -- (1,1) -- (2,0) -- (3,1) -- (4,0) -- (5,1) -- (6,2) -- (7,3) -- (8,2) -- (9,3) -- (10,4) ;
  
\filldraw[red] (0,0) circle (2.5pt);
\filldraw[blue] (1,1) circle (2.5pt);
\filldraw[red] (2,0) circle (2.5pt);
\filldraw[blue] (3,1) circle (2.5pt);
\filldraw[red] (4,0) circle (2.5pt);
\filldraw[blue] (5,1) circle (2.5pt);
\filldraw[blue] (6,2) circle (2.5pt);
\filldraw[blue] (7,3) circle (2.5pt);
\filldraw[blue] (8,2) circle (2.5pt);
\filldraw[blue] (9,3) circle (2.5pt);
\filldraw[blue] (10,4) circle (2.5pt);
    \end{tikzpicture}
    \caption{Left picture: number of contacts (in red)\ for the simple random walk on the integer lattice $\mathbb Z$. Right picture: number of contacts for a reflected simple random walk on $\mathbb N_0$}
    \label{fig:RW_dim1}
\end{center}    
\end{figure}

Our main models (to be presented in Sections~\ref{sec:model-1} and \ref{sec:model-2} below)\ are inspired by the above innocent example. We will consider random walks in the quarter plane with reflection conditions on the boundary axes (see Figure~\ref{fig:RW_dim2}) and the associated local times on the axes. Assuming that the models' drift belongs to the cone ensures that the local time is finite. We will then demonstrate how the distribution of local time can be calculated using an infinite sum of geometric terms. Before presenting our results, we will outline several reasons why this local time problem is important.

\begin{figure}[h!]
\begin{center}
  \begin{tikzpicture}[
    x=5.5mm,
    y=5.5mm,
  ]

    \draw [fill=white,gray!20,draw opacity=1] (0,0) rectangle (11,-3);
    \draw[thin]
      \foreach \x in {0, ..., 10} {
        (\x, -3) -- (\x, 6)
      }
      \foreach \y in {-2, ..., 5} {
        (0, \y) -- (11, \y)
      }
    ;  

    \begin{scope}[
      semithick,
      ->,
      >={Stealth[]},
    ]
      \draw (0, -3.5) -- (0, 6.5);
      \draw (0, 0) -- (11.5, 0);
    \end{scope}
  \draw[thick] (0,3) -- (1,4) -- (2,5) -- (3,4) -- (4,3);
 \draw[thick] -- (4.05,3) -- (5.05,2);
 \draw[thick] (5,2) -- (6,3) -- (7,2);
 \draw[thick] (7.05,2) -- (8.05,1) -- (9.05,0) -- (10.05,1)  ;
  \draw[thick,brown] (0,1) -- (1,2) -- (2,3) -- (3,2) -- (4,3);
  \draw[thick,brown] (3.95,3) -- (4.95,2);
  \draw[thick,brown] (5,2) -- (6,1) -- (7,2);
  \draw[thick,brown] (6.95,2) -- (7.95,1) -- (8.95,0) -- (9.95,1) ;
\filldraw[blue] (0,3) circle (2.5pt);
\filldraw[blue] (0,1) circle (2.5pt);
\filldraw[blue] (1,2) circle (2.5pt);
\filldraw[blue] (1,4) circle (2.5pt);
\filldraw[blue] (2,5) circle (2.5pt);
\filldraw[blue] (3,4) circle (2.5pt);
\filldraw[red] (4,3) circle (2.5pt);
\filldraw[red] (5,2) circle (2.5pt);
\filldraw[blue] (6,3) circle (2.5pt);
\filldraw[red] (7,2) circle (2.5pt);
\filldraw[red] (8,1) circle (2.5pt);
\filldraw[blue] (2,3) circle (2.5pt);
\filldraw[blue] (3,2) circle (2.5pt);
\filldraw[blue] (6,1) circle (2.5pt);
\filldraw[red] (9,0) circle (2.5pt);
\filldraw[red] (10,1) circle (2.5pt);
    \end{tikzpicture}\quad
  \begin{tikzpicture}[
    x=5.5mm,
    y=5.5mm,
  ]

     \draw [fill=white,gray!20,draw opacity=1] (-3,-3) rectangle (6,-1);
      \draw [fill=white,gray!20,draw opacity=1] (-3,-3) rectangle (-1,6);
    \draw[thin]
      \foreach \x in {-3, ..., 6} {
        (\x, -3) -- (\x, 6)
      }
      \foreach \y in {-3, ..., 6} {
        (-3, \y) -- (6, \y)
      }
    ;  

    \begin{scope}[
      semithick,
      ->,
      >={Stealth[]},
    ]
   
      \draw (-1, -3.5) -- (-1, 6.5);
      \draw (-3.5,-1) -- (6.5, -1);
    \end{scope}
  \draw (-1,-1) -- (0,0) -- (1,-1) -- (2,0) -- (1,1)  -- (0,2) -- (-1,3) -- (0,4) -- (-1,5) -- (0,6) -- (1,5) -- (2,4) -- (3,5) -- (4,4) -- (5,5) ;
  
\filldraw[red] (-1,-1) circle (2.5pt);
\filldraw[blue] (0,0) circle (2.5pt);
\filldraw[red] (1,-1) circle (2.5pt);
\filldraw[blue] (2,0) circle (2.5pt);
\filldraw[blue] (1,1) circle (2.5pt);
\filldraw[blue] (0,2) circle (2.5pt);
\filldraw[red] (-1,3) circle (2.5pt);
\filldraw[blue] (0,4) circle (2.5pt);
\filldraw[red] (-1,5) circle (2.5pt);
\filldraw[blue] (0,6) circle (2.5pt);
\filldraw[blue] (1,5) circle (2.5pt);
\filldraw[blue] (2,4) circle (2.5pt);
\filldraw[blue] (3,5) circle (2.5pt);
\filldraw[blue] (4,4) circle (2.5pt);
\filldraw[blue] (5,5) circle (2.5pt);

    \end{tikzpicture}
    \caption{Left: two ordered simple random walks on $\mathbb{N}_0$, exhibiting two types of boundary contacts: coincidences between the upper and lower paths, and contacts with the horizontal axis. Right: A singular random walk in the quarter plane $\mathbb N_0^2$, with jumps $(-1,1)$, $(1,-1)$, and $(1,1)$; boundary contacts are indicated in red.}
    \label{fig:RW_dim2}
\end{center}    
\end{figure}

\paragraph{Probability theory.}

The local time at a given point, as discussed above, plays a fundamental role in the theory of random walks. It is closely related to several key quantities, including the Green function (through the number of visits), return times, and more broadly, to fluctuation theory. For a classical probabilistic treatment of these topics, we refer the reader to Chapter~3 of \cite{Fe-68}. For recent applications involving the local time in neighborhoods of the boundary (particularly in the context of branching processes in random environments) see \cite{ElWa-25}.

\paragraph{Statistical physics.}

As noted in \cite{BeOwRe-19}, part of the motivation for this study originates from \cite{TaOwRe-14,TaOwRe-16}, which analyses models of interacting directed polymers. In \cite{TaOwRe-14}, the authors investigate a model involving two interacting polymers constrained above an impenetrable wall (see the left side of Figure~\ref{fig:RW_dim2}). The model consists of two paths of equal length confined to the upper half-space, with each step moving either up or down. One path is designated as the top path and is constrained never to cross below the other.
The model assigns specific weights to different types of vertex visits: vertices on the horizontal boundary visited by the bottom path receive one type of weight; vertices visited by both paths receive another distinct weight.
There exists a natural correspondence between this model and a class of interacting walks confined to the quarter plane (illustrated on the right side of Figure~\ref{fig:RW_dim2}).

The model examined in \cite{TaOwRe-16} is a variation: instead of two non-crossing paths above a wall, it considers three non-crossing paths without a wall. By tracking the distances between the top and middle paths, and between the middle and bottom paths, the system can be mapped onto a single path in two dimensions. The non-crossing condition ensures that this corresponding path remains confined to the quarter plane.

From a physical perspective, the partition function can be defined as the normalisation of the Boltzmann distribution on walks of length $n$. This function depends on the weights assigned to visits to the bottom (or left) boundary. The limiting free energy of the model is then defined as the exponential growth rate of the partition function as $n \to \infty$. The free energy captures the typical, large-scale behavior of the walks and characterizes the different phases of the model. See \cite{BrOwReWh-05,OwReWo-12,BeOwXu-21,BoKaVe-21} for related literature.

\paragraph{Analytic combinatorics of lattice paths.}
Lattice paths have been the subject of intensive study within the combinatorial community. In particular, in the context of the present paper, Theorem~5 in \cite{BaFl-02} computes the number of boundary contacts for excursions whose length goes to infinity. Here, excursions refer to walks that are constrained to be non-negative and have fixed starting and ending points. The resulting limit law is a sum of two geometric distributions. The paper \cite{BaWa-14} (see in particular Theorem~4.2) studies other models, such as random walks on $\mathbb N_0$ reflected at $0$, and derives various limit laws, such as that of the number of touches with $0$. See \cite{Wa-20,BaKuWa-24} for related theorems on phase transitions in random walks.

\paragraph{Main results.}
In Section~\ref{sec:model-1}, we consider a random walk in the quarter plane with identical horizontal and vertical reflections. In Theorem~\ref{thm: f_n(x) recurrence relation}, we present a recursive method to compute the probability of observing a given number of  contacts, starting from an arbitrary point within the cone. The assumption on the reflections plays a crucial role in the proof, enabling the use of a coupling argument. In Proposition~\ref{prop:lower_upper}, we further derive the precise tail asymptotics of the boundary contact distribution.

In Section~\ref{sec:model-2}, we relax the assumption on the reflections by considering a model with arbitrary horizontal and vertical reflections. On the other hand, we impose a singularity condition on the random walk, requiring that it cannot jump in the directions $(-1,0)$, $(-1,-1)$, or $(0,-1)$. Under this setting, Theorem~\ref{thm: G(i,j,z) generating function} provides a closed-form expression for the boundary contact distribution via a compensation method, an approach that fundamentally relies on the singularity assumption (refer to the next paragraph for more details on the technique). We also derive several asymptotic estimates for the distribution, both as the starting point tends to infinity and as the number of contacts increases; see e.g.\ Proposition~\ref{prop: f_n asymptotic}.

In Section~\ref{sec:examples}, we present several illustrations of our results on particularly relevant models.

\paragraph{Compensation approach.}

As a side note, our work introduces a new instance of the applicability of the compensation approach, specifically in computing the distribution of the number of boundary contacts for random walks in the quarter plane.

The compensation approach has been developed within the probabilistic framework of stationary distributions for random walks; see, for example, \cite{Ad-91,AdWaZi-90,AdWeZi-93,AdBoRe-01,AvVLRa-13}. Unlike traditional methods for quadrant walk problems—which typically aim to determine a generating function—the compensation approach focuses instead on directly computing the coefficients, which in our case correspond to the probabilities of observing a given number of contacts.

The key idea is that the number of contacts, viewed as a function of three variables (two spatial coordinates and one time parameter), satisfies a natural recurrence relation. Using the compensation approach, we will solve this recurrence in the form of infinite sums of elementary geometric terms.

\section{A first model with same horizontal and vertical reflections}
\label{sec:model-1}

\subsection{The model}
We begin by precisely defining our first model of reflected random walk in the positive quadrant, starting with its increments. Let $\xi$ and $\eta$ be $\mathbb{Z}^2$-valued random variables, described by their probability mass functions (pmfs):
\begin{equation}
\label{eq: xi, eta pms}
\mathbf{P} (\xi = (k,\ell) ) = p_{k,\ell},
\quad
\mathbf{P}(\eta = (k,\ell) ) = q_{k,\ell},
\quad
(k,\ell) \in \mathbb{Z}^2.
\end{equation}
The random variables $\xi$ and $\eta$ represent the step increments of the walk in the interior $\mathbb{N}^2$ and on the boundary $\mathbb{N}_0^2 \setminus \mathbb{N}^2$ of the quadrant, respectively. The sets of weights $\{p_{k,\ell}\}_{k,\ell}$ and $\{q_{k,\ell}\}_{k,\ell}$ are non-negative and satisfy the following conditions:

\noindent\begin{minipage}{0.63\textwidth}
\vspace{-5.1cm}
\begin{enumerate}[label=(A\arabic{*}),ref={\rm (A\arabic{*})}]
	\item\label{asm: A1}  $\sum_{k,\ell} p_{k,\ell} = \sum_{k,\ell} q_{k,\ell} = 1$  (\textit{normalization})
	\item\label{asm: A2} $p_{k,\ell}=0$ for all $k,\ell \leq -2$  (\textit{small $<0$ jumps})
	\item\label{asm: A3} $p_{k,-1} p_{-1,\ell}>0$ for some $k,\ell$  (\textit{existence $<0$ jumps})
	\item\label{asm: A4} $\sum_{k,\ell} k p_{k,\ell},\sum_{k,\ell} \ell p_{k,\ell} \in(0,\infty)$  (\textit{drift $>0$})
	\item\label{asm: A5} $q_{k,\ell}=0$ for all $k,\ell \leq -1$  (\textit{non-negative jumps})
	\item\label{asm: A6} $\sum_{k,\ell} k q_{k,\ell},\sum_{k,\ell} \ell q_{k,\ell} \in(0,\infty)$  (\textit{boundary drift $>0$})
\end{enumerate}
\end{minipage}
 \begin{tikzpicture}[
    x=5.3mm,
    y=5.3mm,
  ]
\vspace{2.5cm}
     \draw [fill=white,gray!20,draw opacity=1] (-3,-3) rectangle (6,-1);
      \draw [fill=white,gray!20,draw opacity=1] (-3,-3) rectangle (-1,6);
    \draw[thin]
      \foreach \x in {-3, ..., 6} {
        (\x, -3) -- (\x, 6)
      }
      \foreach \y in {-3, ..., 6} {
        (-3, \y) -- (6, \y)
      }
    ;  

    \begin{scope}[
      semithick,
      ->,
      >={Stealth[]},
    ]
   
   \draw[->,thick,blue] (3,3) -- (4,4);
   \draw[->,thick,blue] (3,3) -- (2,3);
   \draw[->,thick,blue] (3,3) -- (3,2);
   \draw[->,thick,blue] (3,3) -- (2,5);
   \draw[->,thick,blue] (3,3) -- (5,4);
   
   \draw[->,thick,red] (1,-1) -- (1,0);
    \draw[->,thick,red] (1,-1) -- (2,1);
    \draw[->,thick,red] (1,-1) -- (4,0);
    
     \draw[->,thick,red] (-1,1) -- (-1,2);
    \draw[->,thick,red] (-1,1) -- (0,3);
    \draw[->,thick,red] (-1,1) -- (2,2);
    
      \draw (-1, -3.5) -- (-1, 6.5);
      \draw (-3.5,-1) -- (6.5, -1);
      \filldraw[black] (0,0) circle (0pt) node[]{\textcolor{red}{$q_{k,\ell}$}};
      \filldraw[black] (4.5,4.5) circle (0pt) node[]{\textcolor{blue}{$p_{k,\ell}$}};
    \end{scope}
    \end{tikzpicture}

Now let $\{\xi_n\}_{n\geq 0}$ and $\{\eta_n\}_{n\geq 0}$ be sequences of independent copies of the random variables $\xi$ and $\eta$, respectively. For any initial point $x \in \mathbb{N}_0^2$, the random walk $\{S_n^x\}_{n\geq 0}$ is defined as follows:
\begin{equation*}
S_0^x = x, \quad S_{n+1}^x = S_{n}^x + \xi_{n+1}\cdot \mathbf{1}_{\{S_n^x\in\mathbb{N}^2\}} + \eta_{n+1}\cdot \mathbf{1}_{\{S_n^x\in\mathbb{N}_0^2\setminus \mathbb{N}^2\}},\quad n\geq 0.
\end{equation*}
Several important properties of the model follow from Assumptions~\ref{asm: A1}--\ref{asm: A6}. First, these assumptions guarantee that the reflected random walk is irreducible on its state space $\mathbb{N}_0^2$. Inside $\mathbb{N}^2$, the walk has a positive drift, so for a starting point sufficiently far from the axes, the probability of remaining forever within the quadrant $\mathbb{N}^2$ lies in the interval $(0,1)$. Another key feature is that on the boundary $\mathbb{N}_0^2 \setminus \mathbb{N}^2$, the transition probabilities are identical along the horizontal and vertical axes. Moreover, the positive drift along the boundary ensures that the walk almost surely reflects back into the interior after some time.

\subsection{Numbers of contacts}

Our aim is to study the time (interpreted as the number of visits, or more precisely, the local time)\ that the walk spends on the boundary before drifting to infinity. This quantity is represented by a random variable $Z_x$, which admits the following representation:
\begin{equation}
\label{eq: definition Z_x}
	Z_x = \sum_{n\geq 0} \mathbf{1}_{\{S_n^x\in\mathbb{N}_0^2\setminus \mathbb{N}^2   \}},\quad x\in\mathbb{N}_0^2.
\end{equation}
For brevity, we will also refer to $Z_x$ as the number of boundary contacts. An alternative way to define $Z_x$ is via a sequence of hitting times. Specifically, for any $n \geq 1$, let $\tau^x_n$ denote the time at which the random walk $\{S_n^x\}_{n \geq 0}$ hits the boundary for the $n$\textsuperscript{th} time, that is,
	\begin{equation}
	\label{eq: definition tau_n^x}
    \left\{\begin{array}{rcll}
	\tau^x_1 &=& \displaystyle \inf \{ k\geq 0: S_k^x\in\mathbb{N}_0^2\setminus \mathbb{N}^2   \},&\smallskip\\ 
    \tau^x_{n+1} &=& \displaystyle\tau_n^x + \inf \{ k\geq 1: S^x_{\tau^x_n+k} \in \mathbb{N}_0^2\setminus \mathbb{N}^2  \}, &\quad n\geq 1.
    \end{array}\right.
\end{equation}
The random variable $Z_x$ can then be characterized in terms of the sequence $\{\tau_n^x\}_{n \geq 1}$ as follows:
\begin{equation}
\label{eq: definition alternative Z_x}
	Z_x = \sum_{n\geq 1} \mathbf{1}_{\{ \tau_n^x<\infty \}}.
\end{equation}
According to definitions \eqref{eq: definition tau_n^x} and \eqref{eq: definition alternative Z_x}, if the walk starts at a point $x$ on the boundary, this initial position is counted as a first contact.

The formulation \eqref{eq: definition alternative Z_x} offers a more convenient framework than the initial definition~\eqref{eq: definition Z_x} for analyzing the probability of events involving $Z_x$. Before moving forward, we first establish that $Z_x$ is finite through the following lemma.

\begin{lemma}
\label{lem: Z_x finite}
	Under Assumptions~\ref{asm: A1}--\ref{asm: A6}, for any $x \in \mathbb{N}_0^2$, we have $Z_x < \infty$ a.s.
\end{lemma}

\begin{proof}
To prove the lemma, it suffices to show that $\lim_{n\to\infty} \mathbf{P} ( Z_x \geq n )=0$ for any point $x \in \mathbb{N}_0^2$. Using the stopping times introduced in \eqref{eq: definition tau_n^x}, we reformulate
\begin{equation*}
    \mathbf{P} ( Z_x \geq n )  = \mathbf{P} (\tau_1^x < \infty, \ldots,\tau_n^x<\infty).
\end{equation*}
To prove that the above probability goes to $0$ as $n\to\infty$, we first consider a starting point $x \in \mathbb{N}_0^2$ located sufficiently far from the origin $(0,0)$ (for instance, with $\vert x \vert_1 \geq 3$), so that there exist paths from $x$ that avoid the boundary axes, since by Assumption~\ref{asm: A4}, the random walk has a drift toward the interior of the quadrant. As a result, if $x \in \mathbb{N}^2$, then $\mathbf{P}(\tau_1^x = \infty) > 0$, and if $x \in \mathbb{N}_0^2 \setminus \mathbb{N}^2$, then $\mathbf{P}(\tau_2^x = \infty) > 0$.

By a monotonicity argument, one can further assert that
\begin{equation*}
\min_{x\in\mathbb{N}^2, \vert x \vert_1 \geq 3} \mathbf{P} (\tau_1^x = \infty )>0\quad \text{and}\quad \min_{x\in\mathbb{N}_0^2\setminus\mathbb{N}^2 ,\vert x \vert_1 \geq 3} \mathbf{P} (\tau_2^x = \infty )>0,
\end{equation*}
and these minima should be reached within a finite neighborhood of the origin.
	
On the other hand, any state $x\in \mathbb{N}_0^2$ located near the origin (e.g., $ \vert x\vert_1  < 3$) is non-absorbing due to irreducibility, so that there is a non-zero probability that the walk exits this neighborhood within at most four contacts with the boundary. Once outside this region, the walk has a positive probability of escaping to infinity while avoiding further boundary interactions. This leads to the conclusion that $\min_{x\in\mathbb{N}_0^2}  \mathbf{P} (\tau_{4}^x = \infty ) > 0$. As a consequence, we have
	 \begin{equation*}
	\max_{x\in\mathbb{N}_0^2}\mathbf{P} ( \tau_1^x<\infty,\ldots,\tau_{4}^x<\infty ) < 1- \min_{x\in\mathbb{N}_0^2}\mathbf{P} ( \tau_{4}^x = \infty)  < 1.
\end{equation*}
We can now conclude:
	\begin{align*}
	\mathbf{P} (\tau_1^x < \infty, \ldots,\tau_n^x<\infty) &  \leq   \prod_{k=0}^{\lfloor n/4 \rfloor-1}\mathbf{P}(\tau_{4k+1}^x<\infty,\ldots,\tau_{4k+4}< \infty | \tau_{4k}^x<\infty )  \\
	&  \leq \left( \max_{x\in\mathbb{N}_0^2} \mathbf{P} ( \tau_1^x<\infty,\ldots,\tau_{4}^x<\infty )\right)^{\lfloor n/4\rfloor} \to 0,
\end{align*}
as $n\to \infty$, where we conventionally set $\tau_0^x=0$, a.s. 
\end{proof}

\subsection{A recursion satisfied by the numbers of contacts}

We now focus on studying the pmfs of the number of contacts, denoted by
\begin{equation}
\label{eq:def_f_n}
	f_n(x) = \mathbf{P}(Z_{x}=n),\quad n\in\mathbb{N}_0,\quad x\in\mathbb{N}_0^2.
\end{equation} 
As shown in the following theorem, $f_n$ satisfies a first-order recursive relation in $n$, enabling the computation of any term $f_{n+1}$ from the previous term $f_{n}$.


\begin{theorem}
\label{thm: f_n(x) recurrence relation}
	Under Assumptions~\ref{asm: A1}--\ref{asm: A6} and for all $x\in\mathbb{N}_0^2$, the sequence of pmfs $\{f_n\}_{n\geq 0}$ satisfies the recurrence relation
	\begin{equation}
	\label{eq: f_n(x) recurrence relation}
	\left\{\begin{array}{rcll}
    f_1 (x) &=& \displaystyle \mathbf{E} f_0 (x+\eta) - f_0 (x), &\smallskip\\
	f_{n+1} (x) &=& \displaystyle \mathbf{E} f_n (x+\eta),&\quad n\geq 1,
    \end{array}\right.
\end{equation}
where $\eta$ is defined in \eqref{eq: xi, eta pms}. As a consequence, the cumulative distribution function of $Z_x$ admits the representation below, with $\{\eta_k\}_{k\geq 1}$ being independent copies of $\eta$,
\begin{equation}
\label{eq: f_n(x) exp}
	\mathbf{P}(Z_x\leq n)=\sum_{k=0}^n f_k(x) = \mathbf{E} f_0 \left( x + \sum_{k=1}^n \eta_k \right),\quad n\geq 1,\quad x\in\mathbb{N}_0^2.
\end{equation}
\end{theorem}

Before proceeding with the proof, let us first discuss some aspects of Theorem~\ref{thm: f_n(x) recurrence relation}. This is a fairly general result, yet it can be derived using elementary arguments based on a coupling technique. It shows that, once the initial values $f_0$ are known, all subsequent values $f_n$ for $n \geq 1$ follow naturally. From a technical standpoint, the result does not rely on any assumptions regarding higher-order moments (other than \ref{asm: A3} and \ref{asm: A6}). This will be contrasted in Section~\ref{sec:model-2}, where a second model will be considered and similar results derived, this time under higher moment assumptions.

However, the result given in Theorem~\ref{thm: f_n(x) recurrence relation} no longer holds if the transition probabilities along the horizontal and vertical axes differ. Moreover, additional techniques are required to derive an explicit expression for $f_0$. The problem of survival probability has been extensively studied, and in several cases, explicit formulas for $f_0$ have been obtained; see \cite[Cor.~8]{KuRa-11}, \cite[Thm~C]{GaRa-14}, and \cite[Cor.~20]{HoRaTa-23}. In such cases, Theorem~\ref{thm: f_n(x) recurrence relation} can then be used to derive explicit expressions for $f_n$ for any $n \geq 1$.

\begin{proof}[Proof of Theorem~\ref{thm: f_n(x) recurrence relation}]
Our arguments to prove Theorem~\ref{thm: f_n(x) recurrence relation} are based on a coupling technique. For any starting points $x,y\in\mathbb{N}_0^2$, let us consider the coupled process 
\begin{equation*}
\bigl( S_n^{x} , S_n^{x+y} \bigr)_{n\geq 0}.
\end{equation*}
It is readily seen that for any $n\leq \tau_1^{x}$, we have $S_n^{x} + y = S_n^{x+y}$ a.s. We further have $S_{\tau_1^{x} + 1}^{x} = S_{\tau_1^{x} }^{x+y}$ a.s.\ under the conditional measure $\mathbf{P}(\cdot | \tau_1^{x}<\infty, \eta_{\tau_1^{x} + 1} = y )$. As a consequence, 
\begin{equation*}
	\mathbf{P} ( S_{\tau_1^{x}+1+n}^{x} \in \cdot | \tau_1^{x}<\infty, \eta_{\tau_1^{x} + 1} = y )  = \mathbf{P} (  S_{\tau_1^{x}+n}^{x+y} \in \cdot | \tau_1^{x}<\infty  )
\end{equation*}
for any $n\geq 1$, and thus for any $n\geq 1$,
\begin{equation*}
	\mathbf{P} ( \tau_{n+1}^{x} \in \cdot | \tau_1^{x}<\infty, \eta_{\tau_1^{x} + 1} = y )  = \mathbf{P} ( \tau_n^{x+y} + 1 \in \cdot | \tau_1^{x}<\infty  ).
\end{equation*}

The above assertion enables us to estimate the number of contacts of a random walk starting at $x$ by comparing it with that of a coupled random walk started at a different point. This comparison can be used to express $f_{n+1}$ in terms of $f_n$. We start by proving the first identify in \eqref{eq: f_n(x) recurrence relation}. For any $x\in\mathbb{N}_0^2$, we have:
\begin{align*}
	f_1 (x) 
	 &= \mathbf{P} ( \tau_1^{x}<\infty, \tau_2^{x}=\infty ) \\
	 &= \sum_{y\in\mathbb{N}_0^2} q_{y} \mathbf{P} (  \tau_2^{x}=\infty | \tau_1^{x}<\infty, \eta_{\tau_1^{x} + 1} = y ) \mathbf{P} ( \tau_1^{x}<\infty ) \\
	 &= \sum_{y\in\mathbb{N}_0^2} q_{y}  \mathbf{P} (  \tau_1^{x+y}+1 =\infty | \tau_1^{x}<\infty )  \mathbf{P} ( \tau_1^{x}<\infty ) \\
	 &= \sum_{y\in\mathbb{N}_0^2} q_{y} \mathbf{P} ( \tau_1^{x}<\infty,   \tau_1^{x+y}=\infty ) \\
	 &= \sum_{y\in\mathbb{N}_0^2} q_{y}  \left(  \mathbf{P} (   \tau_1^{x+y}=\infty )  -  \mathbf{P} ( \tau_1^{x}=\infty,   \tau_1^{x+y}=\infty )\right) \\
	 &= \sum_{y\in\mathbb{N}_0^2} q_{y}  \left(  \mathbf{P} (   \tau_1^{x+y}=\infty )  -  \mathbf{P} ( \tau_1^{x}=\infty )\right) \\
	 &= \sum_{y\in\mathbb{N}_0^2} q_{y}  f_0(x+y ) - f_0(x).
\end{align*}
Similarly, we show the second identity in \eqref{eq: f_n(x) recurrence relation}. For any $n\geq 1$ and $x\in\mathbb{N}_0^2$, we have:
\begin{align*}
	f_{n+1} (x)  &= \mathbf{P} ( \tau_1^{x} < \infty ,\ldots, \tau_{n+1}^{x} < \infty , \tau_{n+2}^{x} = \infty ) \\
	&= \sum_{y\in\mathbb{N}_0^2} q_{y} \mathbf{P} (  \tau_2^{x} < \infty ,\ldots, \tau_{n+1}^{x} < \infty , \tau_{n+2}^{x} = \infty |  \tau_1^{x} < \infty , \eta_{\tau_1^{x}+1} = y ) \mathbf{P} (  \tau_1^{x} < \infty ) \\
	&= \sum_{y\in\mathbb{N}_0^2} q_{y} \mathbf{P} ( \tau_1^{x+y}+1 < \infty ,\ldots , \tau_{n}^{x+y}+1 < \infty, \tau_{n+1}^{x+y}+1 = \infty |  \tau_1^{y} < \infty  ) \mathbf{P} (  \tau_1^{y} < \infty ) \\
	&= \sum_{y\in\mathbb{N}_0^2} q_{y} \mathbf{P} (  \tau_1^{x} < \infty, \tau_1^{x+y} < \infty , \ldots, \tau_{n}^{x+y} < \infty, \tau_{n+1}^{x+y} = \infty )  \\
	&= \sum_{y\in\mathbb{N}_0^2} q_{y} \mathbf{P} (  \tau_1^{x+y} < \infty , \ldots, \tau_{n}^{x+y} < \infty, \tau_{n+1}^{x+y} = \infty ) \\
	&= \sum_{y\in\mathbb{N}_0^2} q_{y} f_n (x+y).
\end{align*}
To conclude, we prove \eqref{eq: f_n(x) exp}. Noting that $\{\eta_n\}_{n\geq 1}$ are independent copies of 
$\eta$, a simple induction yields
\begin{align*}
    &f_1 (x) = \mathbf{E} f_0 (x+\eta_1) - f_0 (x),\\
    &f_n (x) = \mathbf{E} f_{n-1} (x+\eta_n) = \mathbf{E} f_0 \left(x+\sum_{k=1}^n \eta_k \right) - \mathbf{E} f_0 \left(x+ \sum_{k=2}^n \eta_k\right),\quad n\geq 2,\quad x\in\mathbb{N}_0^2.
\end{align*}
The proof is complete.
\end{proof}

\subsection{Upper and lower bounds for the tail probability}
Rather than pursuing results on explicit expressions at this stage, we proceed to an independent consequence of Theorem~\ref{thm: f_n(x) recurrence relation}: deriving upper and lower bounds for the tail probability $\mathbf{P}(Z_x>n)$. 

Let us introduce some notation before stating our result. Given two real sequences $\{a_n\}_{n\geq 0}$ and $\{b_n\}_{n\geq 0}$, we write $a_n \asymp b_n$ as $n \to \infty$ if there exist constants $c_1, c_2 > 0$ such that $c_1 \leq a_n / b_n \leq c_2$ for all sufficiently large $n$. We also denote the maximum and minimum of any two real numbers $a$ and $b$ by
\begin{equation*}
a \vee b = \max\{a,b\}\quad\text{and}\quad
a \wedge b = \min\{a,b\}.
\end{equation*}
The following proposition is a consequence of Theorem~\ref{thm: f_n(x) recurrence relation}, describing the asymptotic behavior of the tail probability $\mathbf{P}(Z_x > n)$. Denote by $\alpha_1$ and $\beta_{-1}$ the unique solutions in $(0,1)$ of the equations
    \begin{equation}\label{eq: alpha_1, beta_-1 equations}
		\alpha = \sum_{k,\ell} p_{k,\ell} \alpha^{k+1}\quad \text{and}\quad 
		\beta = \sum_{k,\ell} p_{k,\ell} \beta^{\ell+1}.
\end{equation}
Let us further denote by $\Pi_1$ and $\Pi_2$ the projections of any point in $\mathbb{R}^2$ onto the first and second coordinates, respectively.
\begin{proposition}
\label{prop:lower_upper}
	Under Assumptions \ref{asm: A1}--\ref{asm: A6}, for any $x=(i,j)\in\mathbb{N}_0^2$, we have as $n\to\infty$
	\begin{equation*}
		\mathbf{P} (Z_{x} >n ) \asymp \alpha_1^i \left(\mathbf{E}\alpha_1^{ \Pi_1 (\eta) } \right)^n  +\beta_{-1}^j \left( \mathbf{E} \beta_{-1}^{ \Pi_2 (\eta) } \right)^n,
\end{equation*}
where $\alpha_1,\beta_{-1}$ are introduced in \eqref{eq: alpha_1, beta_-1 equations}, and $\eta$ is defined in \eqref{eq: xi, eta pms}.
\end{proposition}

These bounds follow from corresponding upper and lower estimates of $f_0(x)$, the probability of not touching the boundary (see \eqref{eq:def_f_n}), combined with a one-dimensional argument based on the boundary contacts of the projections of the random walk onto the axes. More specifically, we analyze the first hitting times of the random walk to each axis, in the absence of reflected jumps on the axes. We begin by defining a new random walk, without reflection jumps. Let
\begin{equation}\label{eq: walk S'_n}
    S'_0 = (0,0),\quad S'_n = \sum_{k=1}^n \xi_k,\quad n\geq 1,
\end{equation}
where $\{\xi_n\}_{n\geq 1}$, as above, are independent copies of $\xi$, whose pmf is given by \eqref{eq: xi, eta pms}. Then the first hitting times of the random walk $\{x+S'_n\}_{n\geq 0}$ to the vertical and horizontal axes, respectively, are defined as follows, for any $x\in\mathbb{N}_0^2$:
\begin{equation*}
    \widehat{\tau}^x = \inf\bigl\{n\geq 0: \Pi_1 \bigl( x+ S_n' \bigr) =0\bigr\}\quad\text{and}\quad 
    \widetilde{\tau}^x = \inf\bigl\{n\geq 0: \Pi_2 \bigl( x+ S_n' \bigr) =0\bigr\}.
\end{equation*}

\begin{lemma}
\label{lem: formula hitting axes}
    Given any $x = (i,j) \in \mathbb{N}_0^2$, and with $\alpha_1$ and $\beta_{-1}$ as defined in \eqref{eq: alpha_1, beta_-1 equations}, we have
    \begin{equation*}
        \mathbf{P} (\widehat{\tau}^x<\infty) = \alpha_1^i\quad \text{and}\quad
        \mathbf{P} (\widetilde{\tau}^x<\infty) = \beta_{-1}^j.
    \end{equation*}
\end{lemma}
Our notation $\alpha_1$ and $\beta_{-1}$ will become clearer in the following section, where we demonstrate that they are elements of a sequence $\{(\alpha_n,\beta_n)\}_{n\in\mathbb{Z}}$ naturally attached to the model, to be introduced in \eqref{eq: alpha_m beta_m definition}.

\begin{proof}
Lemma~\ref{lem: formula hitting axes} is a classical result in the theory of random walks. Indeed, the problem effectively reduces to a one-dimensional random walk, since only one coordinate plays a role. For completeness, we briefly sketch the proof, as it provides a useful connection with the methods employed in this paper.

Let us first fix $x=(i,j)\in\mathbb{N}_0^2$. We will only prove the result for $\widehat{\tau}^{x}$, as the argument for $\widetilde{\tau}^{x}$ is analogous. For clarity, we omit the value of $j$, indicating it with a dot. Define $\alpha_1=\mathbf{P} (\widehat{\tau}^{1,\cdot}<\infty)$, which implies that $\mathbf{P} (\widehat{\tau}^{i,\cdot}_1<\infty)= \alpha_1^i$ for any $i\geq 1$, since under Assumption \ref{asm: A2}, the corresponding walk is skip-free (or left-continuous). Applying the Markov property yields
\begin{equation*}
\mathbf{P} (\widehat{\tau}^{1,\cdot}<\infty) = \sum_{k,\ell} p_{k,\ell} \mathbf{P} (\widehat{\tau}^{1+k,\cdot}<\infty),
\end{equation*}
which corresponds to the first equation of \eqref{eq: alpha_1, beta_-1 equations}, once we replace $\mathbf{P} (\widehat{\tau}^{k,\cdot}<\infty)$ by $\alpha_1^k$. 

The remaining step is to establish the uniqueness of $\alpha_1$ as a solution to the first equation in \eqref{eq: alpha_1, beta_-1 equations} within the interval $(0,1)$. Define
\begin{equation*}
\widehat{K} (\alpha) = \alpha - \sum_{k,\ell} p_{k,\ell} \alpha^{k+1},\quad \alpha\in [0,1].
\end{equation*}
Since $\widehat{K}''(\alpha) < 0$ for all $\alpha\in (0,1)$, and noting that $\widehat{K}'(0)>0$, $\widehat{K}'(1)<0$, $\widehat{K}(0)<0$, and $\widehat{K}(1)=0$, we conclude that $\widehat{K}(\alpha)$ is strictly concave and attains a maximum at a unique point in $(0,1)$. This ensures that $\alpha_1$ is the only root of $\widehat{K}(\alpha)$ in $(0,1)$. The proof is thus complete.\end{proof}

\begin{proof}[Proof of Proposition~\ref{prop:lower_upper}]
Let us first fix $x=(i,j)\in\mathbb{N}_0^2$. Since $\tau_1^x = \widehat{\tau}^x \wedge \widetilde{\tau}^x$ a.s., then
 \begin{equation*}
     \mathbf{P} \bigl(\widehat{\tau}^{x} < \infty \bigr) \vee \mathbf{P} \bigl(\widetilde{\tau}^{x} < \infty \bigr)  \leq \mathbf{P} \bigl(\tau_{1}^{x} < \infty \bigr)   \leq \mathbf{P} \bigl(\widehat{\tau}^{x} < \infty \bigr) + \mathbf{P} \bigl(\widetilde{\tau}^{x} < \infty \bigr),
 \end{equation*}
 or equivalently, using Lemma~\ref{lem: formula hitting axes},
 \begin{equation}\label{eq: f_0 upper lower}
    \alpha_1^i \vee \beta_{-1}^j \leq 1- f_0(i,j) \leq \alpha_1^i + \beta_{-1}^j.
 \end{equation}
This immediately yields
  \begin{equation*}
     \frac{1}{2} \bigl( \alpha_1^i + \beta_{-1}^j \bigr)  \leq 1- f_0 (i,j) \leq \alpha_1^i + \beta_{-1}^j.
 \end{equation*}
Applying the above inequality to the expectation of the sum $(i,j) + \sum_{k=1}^n\eta_k$ with $n\geq 1$, we have
\begin{align*}
	\frac{1}{2} \left ( \mathbf{E}\alpha_1^{i+\sum_{k=1}^n \Pi_1 (\eta_k) }  + \mathbf{E}\beta_{-1}^{j+\sum_{k=1}^n \Pi_2(\eta_k)  } \right) &\leq 1-\mathbf{E}	f_0\left( (i,j)+ \sum_{k=1}^n\eta_k  \right) \\
    &\leq  \mathbf{E}\alpha_1^{i+\sum_{k=1}^n \Pi_1 (\eta_k) }  + \mathbf{E}\beta_{-1}^{j+\sum_{k=1}^n \Pi_2(\eta_k)  },
\end{align*}
or equivalently with \eqref{eq: f_n(x) exp} and thanks to the independence of $\{\eta_k\}_{1\leq k\leq n}$,
\begin{align*}
	\frac{1}{2} \left( \alpha_1^i \left(\mathbf{E}\alpha_1^{ \Pi_1 (\eta) } \right)^n  +\beta_{-1}^j \left( \mathbf{E} \beta_{-1}^{ \Pi_2 (\eta) } \right)^n \right) &\leq 1- \sum_{k=0}^n f_k(i,j) \\
    &\leq \alpha_1^i \left(\mathbf{E}\alpha_1^{ \Pi_1 (\eta) } \right)^n  +\beta_{-1}^j \left( \mathbf{E} \beta_{-1}^{ \Pi_2 (\eta) } \right)^n.\qedhere
\end{align*}
\end{proof}

\section{A second model: singular random walks with arbitrary boundary reflections}
\label{sec:model-2}

\subsection{The model}

In this section, we investigate a class of reflected random walks in which the transition probabilities along the horizontal and vertical axes may differ. An additional constraint is imposed in the interior of the quadrant: the walk is not allowed to jump closer to the origin. Such processes are known as singular random walks. We now give a precise definition of this class. To maintain notational consistency throughout the paper, we largely retain the symbols introduced in the previous section, adapting them as needed to the present setting.

Let $\xi$, $\eta^\mathbf{h}$, $\eta^\mathbf{v}$ and $\eta^\mathbf{o}$ be random variables in $\mathbb{Z}^2$ characterized by their pmfs (the letters ``$\mathbf{h}$'', ``$\mathbf{v}$'', and ``$\mathbf{o}$'' standing for horizontal, vertical, and origin respectively): for all $( k,\ell) \in\mathbb{Z}^2$,
\begin{equation*}
    \mathbf{P} (\xi = (k,\ell) ) = p_{k,\ell},\
    \mathbf{P}(\eta^\mathbf{h} = (k,\ell) ) = h_{k,\ell},
    \
    \mathbf{P}(\eta^\mathbf{v} = (k,\ell) ) = v_{k,\ell},
    \
    \mathbf{P}(\eta^\mathbf{o} = (k,\ell) ) = q_{k,\ell}
   .
\end{equation*}
The set of weights $\{p_{k,\ell}\}_{k,\ell}$, $\{h_{k,\ell}\}_{k,\ell}$, $\{v_{k,\ell}\}_{k,\ell}$, and $\{q_{k,\ell}\}_{k,\ell}$ are non-negative and satisfy the following conditions:

\bigskip

\noindent\begin{minipage}{0.62\textwidth}
\begin{enumerate}[label=(B\arabic{*}),ref={\rm (B\arabic{*})}]
	\item\label{asm: B1}  $\sum p_{k,\ell} = \sum h_{k,\ell} =\sum v_{k,\ell}=\sum q_{k,\ell} = 1$ \\ (\textit{normalization})
	\item\label{asm: B2} $p_{k,\ell}=0$ for all $k,\ell \leq -2$  (\textit{small $<0$ jumps})
	\item\label{asm: B3} $p_{-1,-1} = p_{-1,0} = p_{0,-1}= 0$  (\textit{singular walks})
    \item\label{asm: B4} $p_{-1,1}p_{1,-1}\not=0$ (\textit{non-degeneracy})
    \item\label{asm: B5} $p_{k,\ell}>0$ for some $(k,\ell)$ with $k+\ell>0$
    \item\label{asm: B6} $\sum_{k,\ell}p_{k,\ell}e^{kx+\ell y}$ is finite in a neighborhood of any point of the curve \eqref{eq: level set L} (\textit{moment assumption})
    \item\label{asm: B7} $h_{k,\ell}=v_{k,\ell}=q_{k,\ell}=0$ if $k\leq 0$ and $\ell \leq 0$ 
    
    (\textit{$\geq 0$ jumps})
    \item\label{asm: B8} $0<\sum_{k,\ell}\ell h_{k,\ell}< \infty$ and $\sum_{k,\ell}k h_{k,\ell}< \infty$; \\$0<\sum_{k,\ell}k v_{k,\ell}< \infty$ and $ \sum_{k,\ell} \ell v_{k,\ell}< \infty$;\\ $0<\sum_{k,\ell}k q_{k,\ell}< \infty$ and $0<\sum_{k,\ell}\ell q_{k,\ell}< \infty$ 
    
    (\textit{boundary first moments and drifts})
\end{enumerate}
\vspace{40mm}
\end{minipage}
\vspace{0mm}
 \begin{tikzpicture}[
    x=5.4mm,
    y=5.4mm,
  ]
\vspace{2cm}
     \draw [fill=white,gray!20,draw opacity=1] (-3,-3) rectangle (6,-1);
      \draw [fill=white,gray!20,draw opacity=1] (-3,-3) rectangle (-1,6);
    \draw[thin]
      \foreach \x in {-3, ..., 6} {
        (\x, -3) -- (\x, 6)
      }
      \foreach \y in {-3, ..., 6} {
        (-3, \y) -- (6, \y)
      }
    ;  

    \begin{scope}[
      semithick,
      ->,
      >={Stealth[]},
    ]
   
   \draw[->,thick,blue] (3,3) -- (4,4) node[ above right ]{$p_{k,\ell}$};
   \draw[->,thick,blue] (3,3) -- (2,4);
   \draw[->,thick,blue] (3,3) -- (4,2);
      \draw[->,thick,blue] (3,3) -- (5,3);
       \draw[->,thick,blue] (3,3) -- (4,5);
   
   \draw[->,thick,red] (2,-1) -- (3,-1);
    \draw[->,thick,red] (2,-1) -- (3,0) node[right]{$h_{k,\ell}$};
    \draw[->,thick,red] (2,-1) -- (3,1);
    \draw[->,thick,red] (2,-1) -- (2,0);
    
    \draw[->,thick,red] (-1,2) -- (-1,3);
    \draw[->,thick,red] (-1,2) -- (0,4) node[above]{$v_{k,\ell}$};
    \draw[->,thick,red] (-1,2) -- (2,3);

    \draw[->,thick,OliveGreen] (-1,-1) -- (-1,0);
    \draw[->,thick,OliveGreen] (-1,-1) -- (1,1) node[above]{$q_{k,\ell}$};
    \draw[->,thick,OliveGreen] (-1,-1) -- (1,-1);
    
      \draw (-1, -3.5) -- (-1, 6.5);
      \draw (-3.5,-1) -- (6.5, -1);
    \end{scope}
    \end{tikzpicture}

\vspace{-40mm}

Let $\{\xi_n\}_{n\geq 0}$, $\{\eta_n^\mathbf{h}\}_{n\geq 0}$, $\{\eta_n^\mathbf{v}\}_{n\geq 0}$, and $\{\eta_n^\mathbf{o}\}_{n\geq 0}$ be sequences of independent copies of $\xi$, $\eta^\mathbf{h}$, $\eta^\mathbf{v}$, and $\eta^\mathbf{o}$, respectively. For an arbitrary $x\in \mathbb{N}_0^2$, the random walk $\{S_n^x\}_{n\geq 0}$ is now defined by $S_0^x = x$ and
\begin{multline*}
    S_{n+1}^x ={} S_{n}^x + \xi_{n+1}\cdot \mathbf{1}_{\bigl\{ S_n^x \in \mathbb{N}^2 \bigr\}} + \eta_{n+1}^\mathbf{h}\cdot \mathbf{1}_{\bigl\{ S_n^x \in \mathbb{N} \times \{0\} \bigr\}} \\
     + \eta_{n+1}^\mathbf{v}\cdot \mathbf{1}_{\bigl\{ S_n^x \in \{0\}\times\mathbb{N} \bigr\}} + \eta_{n+1}^\mathbf{o}\cdot \mathbf{1}_{\bigl\{ S_n^x =(0,0) \bigr\}},\quad n\geq 0.
\end{multline*}
Since the random walk $\{S_n^x\}_{n\geq 0}$ with $x\neq (0,0)$ is singular and thus cannot reach the origin, the final term $\eta_{n+1}^\mathbf{o}\cdot \mathbf{1}_{\left\{ S_n^x =(0,0) \right\}}$ in the definition of $S_n^x$ can be omitted as soon as $x\not=0$. As a result, the analysis of the number of boundary contacts $Z_x$ for starting points $x \neq (0,0)$ does not depend on the specific assumptions made on the weight set $\{q_{k,\ell}\}_{k,\ell}$. Consequently, in this section, we will restrict our attention to random walks starting at $x \neq (0,0)$. The corresponding results for the initial point $x = (0,0)$ will then follow directly from these.


The model presented above can be viewed as a variation of the singular models originally studied in the queueing systems literature \cite{Ad-91,AdWaZi-90,AdWeZi-93,AdBoRe-01}. In these works, the focus lies on the stationary distribution of such models; to render the model stationary, it is necessary to reverse the random walk, that is, to replace the $p_{k,\ell}$ with $p_{-k,-\ell}$. Singular models have also been studied from a combinatorial perspective in \cite{AvVLRa-13,DrHaRoSi-20}. Closer to the present work, harmonic functions associated with singular walks killed at the boundary are analysed in \cite{HoRaTa-23}. In particular, the escape (or survival) probability is computed in closed form. This is closely related to our setting, as the escape probability corresponds to the probability of avoiding the boundary—that is, the probability of having no contact.

As in the previous section, we define the number of boundary contacts $Z_x$ according to \eqref{eq: definition Z_x}, and the sequence of hitting times $\{\tau_n^x\}_{n\geq 1}$ as given in \eqref{eq: definition tau_n^x}. Since the assertion on the almost sure finiteness of $Z_x$ in Lemma~\ref{lem: Z_x finite} does not rely on the identical boundary reflection assumption, it remains valid under the Assumptions~\ref{asm: B1}--\ref{asm: B8}. As in the first model, our goal is to study the pmf $f_n(i,j)$ of the boundary contacts $Z_{(i,j)}$; see \eqref{eq:def_f_n}. However, the coupling technique used previously is no longer applicable, due to the added complexity arising from arbitrary boundary reflections. Instead, we adopt the compensation approach, which will enable us to express the solution through series expansions. In this second model, we thus focus on analysing the generating function of $Z_{(i,j)}$, defined by
\begin{equation}
\label{eq:def_F}
	F(i,j,z)  = \sum_{n\geq 0} f_n(i,j) z^n,\quad (i,j)\in\mathbb{N}_0^2,\quad \vert z\vert\leq 1.
\end{equation}

\subsection{Main result}

For any $i,j\geq 0$, introduce
    \begin{equation}
    \label{eq:def_function_G}
	G(i,j,z) = \sum_{m\in  \mathbb	{Z} }  \bigl( c_m(z) \alpha_m^i \beta_m^j + d_{m+1}(z) \alpha_{m+1}^i \beta_m^j \bigr),
\end{equation} 
where $\left\{ (\alpha_m,\beta_m )\right\}_{m\in\mathbb{Z}}$ is defined in \eqref{eq: alpha_m beta_m definition}, and for any $m\in\mathbb Z$
\begin{equation}\label{eq: c_m d_m definition}
    \left\{\begin{array}{rcl}
        c_0 (z)&=& 1, \\
        d_{m+1}(z) &=& \displaystyle - c_m(z) \frac{1-   z \sum_{k,\ell} v_{k,\ell} \alpha_m^{k} \beta_m^{\ell}}{1 -   z \sum_{k,\ell} v_{k,\ell} \alpha_{m+1}^{k} \beta_m^{\ell}},\smallskip\\
        c_m (z) &=& \displaystyle - d_m (z) \frac{1 -  z \sum_{k,\ell} h_{k,\ell} \alpha_m^{k} \beta_{m-1}^{\ell} }{1 -  z \sum_{k,\ell} h_{k,\ell} \alpha_m^{k} \beta_{m}^{\ell} }.\end{array}\right.
\end{equation}
\begin{theorem}\label{thm: G(i,j,z) generating function}
Under Assumptions~\ref{asm: B1}--\ref{asm: B8}, for any $(i,j)\in\mathbb N_0^2\setminus\{(0,0)\}$, the generating function of $Z_{(i,j)}$ in \eqref{eq:def_F} is equal to $G(i,j,z)$ in \eqref{eq:def_function_G}; in other words, $F(i,j,z) = G(i,j,z)$.
\end{theorem}

For the case $(i,j)=(0,0)$, the generating function $F(0,0,z)$ of $Z_{(0,0)}$ is determined from the expression of other $F(i,j,z)$ via Eq.~\eqref{eq: condition at origin F(i,j,z)}. A large part of the remainder of Section~\ref{sec:model-2} is devoted to the proof of Theorem~\ref{thm: G(i,j,z) generating function}.

\subsection{A bivariate recursion for the boundary contacts generating function}
This subsection aims to formulate a recursion for $F(i,j,z)$, which we will solve using the compensation approach. In Lemma~\ref{lem: recurrence relation f_n(i,j)}, we first establish a bivariate recursion for the pmf $f_n(i,j)$. Building on this, we derive in Proposition~\ref{prop: recurrence relation F(i,j,z)} the corresponding recursive expressions for $F(i,j,z)$.

\begin{lemma}
\label{lem: recurrence relation f_n(i,j)}
	Under Assumption~\ref{asm: B1}--\ref{asm: B8}, the function $f_n:\mathbb{N}_0^2\to [0,1]$ in \eqref{eq:def_f_n} satisfies the recurrence relation
	\begin{equation}\label{eq: recurrence relation f_n(i,j)}
		f_n(i,j) = \sum_{k,\ell} p_{k,\ell}f_n(i+k,j+\ell) , \quad (i,j)\in\mathbb{N}^2,\quad n\geq 0,
\end{equation}
with the horizontal (boundary)\ condition
\begin{equation}
\label{eq: horizontal condition f_n(i,j)}
	f_n (i,0) =\left\{ \begin{array}{ll}
	0 & \text{if } n=0 \\
	\sum_{k,\ell} h_{k,\ell} f_{n-1} ( i+k, \ell)&  \text{if } n\geq 1\end{array}\right\},
\quad i\in\mathbb N,
\end{equation}
the vertical (boundary)\ condition
\begin{equation}
\label{eq: vertical condition f_n(i,j)}
	f_n (0,j) = \left\{\begin{array}{ll}
	0 & \text{if }  n=0 \\
	\sum_{k,\ell} v_{k,\ell} f_{n-1} ( k, j+ \ell) & \text{if }  n\geq 1
\end{array}\right\},
\quad j\in\mathbb N,
\end{equation}
and the condition at the origin
\begin{equation}
\label{eq: condition at origin f_n(i,j)}
	f_n (0,0) = \left\{\begin{array}{ll}
	0 & \text{if }  n=0 \\
	\sum_{k,\ell} q_{k,\ell} f_{n-1} ( k, \ell) & \text{if }  n\geq 1
\end{array}\right\}.
\end{equation}
\end{lemma}

\begin{proof}
Our reasoning is straightforward and relies on decomposing the random walk according to its first step. Observe that for any $(i,j)\in \mathbb{N}^2$ and all $k,\ell\geq -1$, we have $\mathbf{P} ( Z_{(i,j)}\in\cdot | S_1^{(i,j)}=(i+k,j+\ell)) = \mathbf{P} ( Z_{(i+k,j+\ell)}\in\cdot ) $. As a consequence, for any $n\geq 0$,
\begin{align*}
	f_n(i,j) &=  \sum_{k,\ell} p_{k,\ell}  \mathbf{P} \left(Z_{(i,j)} =n \left| S_1^{(i,j)} =(i+k,j+\ell)\right. \right) \\
    &= \sum_{k,\ell} p_{k,\ell}  \mathbf{P} \left(Z_{(i+k,j+\ell)} =n \right)=   \sum_{k,\ell} p_{k,\ell} f_n (i+k,j+\ell),
\end{align*}
which establishes \eqref{eq: recurrence relation f_n(i,j)}.
Now, given that for any $i>0$, $\mathbf{P} ( Z_{(i,0)}\in\cdot | S_1^{(i,0)}=(i+k,\ell)) = \mathbf{P} ( 1+ Z_{(i+k,\ell)}\in\cdot ) $, we have
\begin{equation*}
	f_n(i,0) =  \sum_{k,\ell} h_{k,\ell}  \mathbf{P} \left(Z_{(i,0)} =n \left| S_1^{(i,0)} =(i+k,\ell)\right. \right) = \sum_{k,\ell} h_{k,\ell}  \mathbf{P} \left(1+Z_{(i+k,\ell)} =n \right),
\end{equation*}
which is $0$ if $n=0$, and otherwise equals $\sum_{k,\ell} h_{k,\ell} f_n(i+k,\ell)$, thereby establishing the horizontal condition \eqref{eq: horizontal condition f_n(i,j)}. The vertical condition \eqref{eq: vertical condition f_n(i,j)} and the condition \eqref{eq: condition at origin f_n(i,j)} at the origin follow by analogous arguments.
\end{proof}

\begin{proposition}
\label{prop: recurrence relation F(i,j,z)}
	Under Assumption~\ref{asm: B1}--\ref{asm: B8}, the function $F:\mathbb{N}_0^2\times [0,1]\rightarrow [0,1]$ in \eqref{eq:def_F} satisfies the recurrence relation
	\begin{equation}\label{eq: recurrence relation F(i,j,z)}
		F(i,j,z) = \sum_{k,\ell} p_{k,\ell} F(i+k,j+\ell,z) ,\quad  (i,j)\in\mathbb{N}^2,
	\end{equation}
	with the horizontal (boundary)\ condition
	\begin{equation}
	\label{eq: horizontal condition F(i,j,z)}
	F(i,0,z) = z\sum_{k,\ell} h_{k,\ell} F(i+k,\ell,z),\quad i\in\mathbb N,
\end{equation}
the vertical (boundary)\ condition
	\begin{equation}
	\label{eq: vertical condition F(i,j,z)}
	F(0,j,z) = z\sum_{k,\ell} v_{k,\ell} F(k,j+\ell,z),\quad j\in\mathbb N,
\end{equation}
and the terminal condition
\begin{equation}\label{eq: terminal condition F(i,j,1)}
	F(i,j,1) = 1,\quad (i,j)\in\mathbb{N}_0^2\setminus\{(0,0)\}.
\end{equation}
\end{proposition}

We note that at the origin, the function $F$ satisfies the condition
\begin{equation}\label{eq: condition at origin F(i,j,z)}
    F(0,0,z) = z\sum_{k,\ell} q_{k,\ell} F(k,\ell,z),\quad z\in[0,1].
\end{equation}
This condition is not included in Proposition~\ref{prop: recurrence relation F(i,j,z)} because the system~\eqref{eq: recurrence relation F(i,j,z)}--\eqref{eq: terminal condition F(i,j,1)} is formulated independently of the set of weights $\{q_{k,\ell}\}_{k,\ell}$. The subsequent analysis is thus devoted entirely to solving the system without invoking Eq.~\eqref{eq: condition at origin F(i,j,z)}. 

\begin{proof}
	For any $(i,j)\in\mathbb N^2$, multiplying Eq.~\eqref{eq: recurrence relation f_n(i,j)} by $z^n$ and summing over $n\geq 0$ yields Eq.~\eqref{eq: recurrence relation F(i,j,z)}. Similarly, Eq.~\eqref{eq: horizontal condition F(i,j,z)} is obtained by applying the same procedure to Eq.~\eqref{eq: horizontal condition f_n(i,j)}. The vertical condition \eqref{eq: vertical condition F(i,j,z)} and the condition \eqref{eq: condition at origin F(i,j,z)} at the origin follow by analogous reasoning. Finally, since $Z_{(i,j)} < \infty$ almost surely for all $(i,j)\in \mathbb{N}_0^2$ (see Lemma~\ref{lem: Z_x finite}), Eq.~\eqref{eq: terminal condition F(i,j,1)} follows directly.
\end{proof}

\subsection{Constructing solutions to the recurrence equations}
In this subsection, we apply the compensation approach to construct a family of solutions for the system \eqref{eq: recurrence relation F(i,j,z)}--\eqref{eq: terminal condition F(i,j,1)}. The general idea is outlined as follows. We begin with a base term that satisfies the recurrence relation \eqref{eq: recurrence relation F(i,j,z)}. We then add a new term which still satisfies the recurrence, but corrects the error in the vertical condition \eqref{eq: vertical condition F(i,j,z)}. This correction, however, introduces a new error in the horizontal condition \eqref{eq: horizontal condition F(i,j,z)}, which we address by adding another term. By continuing this alternating process (correcting the vertical condition, then the horizontal one, and so on), we build a series-form solution
\begin{equation}
\label{eq:evocation_sequence}
    \ldots+d_{-1}\alpha_{-1}^i\beta_{-2}^j+c_{-1}\alpha_{-1}^i\beta_{-1}^j + d_0\alpha_0^i\beta_{-1}^j+c_0\alpha_0^i\beta_0^j + d_1\alpha_1^i\beta_0^j + c_1\alpha_1^i\beta_1^j + d_2\alpha_2^i\beta_1^j + \ldots
\end{equation}
where each term satisfies \eqref{eq: recurrence relation F(i,j,z)}, each sum $c_m\alpha_m^i\beta_m^j + d_{m+1}\alpha_{m+1}^i\beta_m^j$ satisfies \eqref{eq: vertical condition F(i,j,z)}, and each sum $d_{m}\alpha_{m}^i\beta_{m-1}^j+c_m\alpha_m^i\beta_m^j$ satisfies \eqref{eq: horizontal condition F(i,j,z)}. As already mentioned, this compensation method has been introduced in the nineties by Adan, Wessels and Zijm in a series of papers \cite{Ad-91,AdWaZi-90,AdWeZi-93}.

The current section is structured around a few preliminary results aimed at constructing the sequence $\{(\alpha_n,\beta_n)\}_{n\in\mathbb{Z}}$, see Lemmas~\ref{lem: level set L}--\ref{lem: alpha_m beta_m decay}. These results will subsequently be used to build the solutions to the recursion in the series form \eqref{eq:evocation_sequence} presented above, and to establish its absolute convergence. Note that the variable $z$ in Equations~\eqref{eq: horizontal condition F(i,j,z)} and \eqref{eq: vertical condition F(i,j,z)} serves solely as a parameter. Accordingly, we assume that $z\in [0,1]$ is fixed throughout the analysis.

Let us begin with the readily verified fact that a product form $\alpha^i \beta^j$ is a solution of \eqref{eq:	 recurrence relation F(i,j,z)} for any $(i,j)\in \mathbb N_0^2$ if and only if $(\alpha,\beta)$ satisfies the equation
\begin{equation*}
	\alpha \beta = \sum_{-1\leq k,\ell} p_{k,\ell} \alpha^{k+1} \beta^{\ell+1}.
\end{equation*}
To study the above equation, let us define the corresponding bivariate polynomial (sometimes called the kernel of the model)
\begin{equation*}
	K(\alpha,\beta) =  \sum_{-1\leq k,\ell} p_{k,\ell} \alpha^{k+1} \beta^{\ell+1} - \alpha\beta, 
\end{equation*}
which is fully characterized by the weights $\{p_{k,\ell}\}_{k,\ell}$.
Its associated level set
\begin{equation}
\label{eq:level_set_K}
   \mathcal{K} = \{ (\alpha,\beta)\in [0,\infty)^2: K(\alpha,\beta) =0\}
\end{equation}
(see Figure~\ref{fig: curves L and K})\ will turn out to be particularly important. The function $K(\alpha,\beta)$ has in fact been thoroughly studied in \cite[Sec.~2]{HoRaTa-23}. Still, we present below a few points that are particularly relevant to our aims. Our starting point in analyzing $K(\alpha,\beta)$ is through its exponential transform, defined for $(x,y)\in\mathbb{R}^2$ by
\begin{equation*}
	L(x,y) = e^{-(x+y)} K(e^x,e^y) =  \sum_{-1\leq k,\ell} p_{k,\ell} e^{xk+y\ell}- 1.
\end{equation*}
We also define the level set of $L(x,y)$ as follows, along with its associated strict sub- and super-level sets:
\begin{gather}
    \label{eq: level set L}\mathcal{L} = \{ (x,y)\in \mathbb{R}^2 : L(x,y) =0\},\\
    \mathcal{L}^- = \{ (x,y)\in \mathbb{R}^2 : L(x,y)<0\}\quad\text{and}\quad
    \mathcal{L}^+ = \{ (x,y)\in \mathbb{R}^2 : L(x,y)>0\}, \notag
\end{gather}
see Figure~\ref{fig: curves L and K}. Throughout the analysis, several properties of $K(\alpha,\beta)$ can be transferred from those of $L(x,y)$, as the latter exponential transform is, in fact, more tractable.

\begin{figure}
    \centering
    \includegraphics[width=0.3\linewidth]{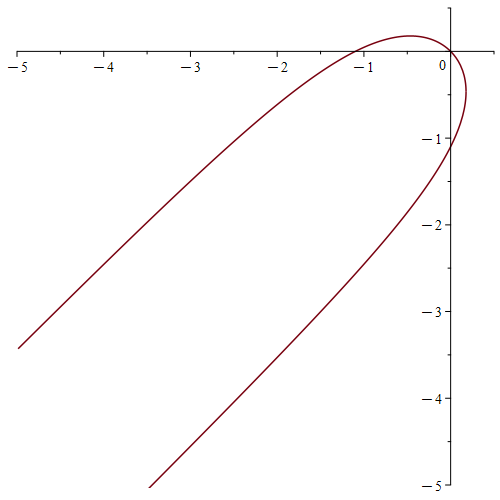}\qquad\qquad
    \includegraphics[width=0.3\linewidth]{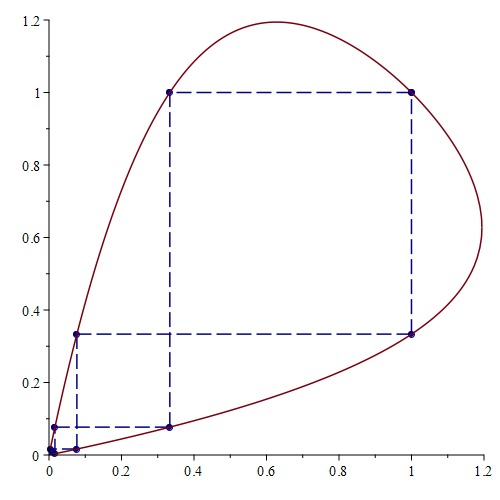}
    \caption{The curves $\mathcal{L}$ (left), $\mathcal{K}$ (right), and the sequence of pairs $\{(\alpha_m,\beta_m)\}_{m\in\mathbb{Z}}$ (blue points) for the random walk with $p_{1,-1} = p_{-1,1} = p_{1,0} = p_{0,1} = p_{1,1} = \tfrac{1}{5}$.  }
    \label{fig: curves L and K}
\end{figure}

\begin{lemma}[\!\cite{HoRaTa-23}]
\label{lem: level set L}
The sublevel set $\mathcal{L}^-$ is convex and contains the ray $\{t(-1,-1):t>0\}$, whereas the superlevel set $\mathcal{L}^+$ includes the rays $\{t(-1,0):t>t_0\}$ and $\{t(0,-1):t>t_0\}$ for some $t_0$ sufficiently large. Additionally, $\mathcal{L}$ admits a tangent at $(0,0)$, which lies outside $\mathcal{L}^-$ and satisfies the equation
\begin{equation*}\textstyle
	\left( \sum_{k,\ell}k p_{k,\ell} \right) x + \left( \sum_{k,\ell}\ell p_{k,\ell} \right) y = 0.
\end{equation*}
\end{lemma}
Lemma~\ref{lem: level set L} is stated and proved in \cite{HoRaTa-23} as Lemma~1. 
It entails that the line $\{t(1,1):t\in\mathbb{R}\}$ divides the curve $\mathcal{L}$ into two parts, which can be parametrized as follows. Let us define
\begin{itemize}
    \item $\gamma:(-\infty,0]\rightarrow \mathbb{R}$ such that $\gamma(0)=0$, $L(x,\gamma(x)) =0$ and $\gamma(x) > x$ for any $x< 0$;
    \item $\zeta:(-\infty,0]\rightarrow \mathbb{R}$ such that $\zeta(0)=0$, $L(\zeta(y),y) =0$ and $\zeta(y)> y$ for any $y< 0$.
\end{itemize}
Both $\gamma$ and $\zeta$ are well-defined functions on $(-\infty,0]$, as will be shown in Lemma~\ref{lem: parametrization level set L} below. The curve $\mathcal{L}$ in \eqref{eq: level set L} can then be described as
\begin{equation*}
	\mathcal{L}= \{(x,\gamma(x)):x\leq 0\} \cup \{(\zeta(y),y):y\leq 0\}. 
\end{equation*}
The following result, which can be found as Lemma~2 in \cite{HoRaTa-23}, gives some key properties of $\gamma$ and $\zeta$.
\begin{lemma}[\!\!{\cite{HoRaTa-23}}]
\label{lem: parametrization level set L}
The following assertions hold:
\begin{enumerate}[label=(\roman*)]
    \item\label{item: gamma zeta well defined} On $(-\infty,0]$, the functions $\gamma$ and $\zeta$ are well defined, concave and infinitely differentiable. 
    \item\label{item: gamma monotonic} The function $\gamma$ has a unique maximizer $\widehat{x}<0$, with corresponding value 
    \begin{equation*}
    \widehat{y}=\max_{x\leq 0} \gamma (x) = \gamma (\widehat x)>0.
    \end{equation*} Moreover, $\gamma$ is strictly increasing on $(-\infty,\widehat{x}]$, with $\lim_{x\to -\infty}\gamma(x)=-\infty$, and strictly decreasing on $[\widehat{x},0]$.
    \item\label{item: zeta monotonic} Similarly, $\zeta$ has a unique maximizer $\widetilde{y}<0$, with corresponding value 
    \begin{equation*}
    \widetilde{x}=\max_{y\leq 0} \zeta (y) =\zeta(\widetilde{y})>0.
    \end{equation*}
    Moreover, $\zeta$ is strictly increasing on $(-\infty,\widetilde{y}]$ with $\lim_{y\to -\infty}\zeta(y)=-\infty$, and strictly decreasing on $[\widetilde{y},0]$.
\end{enumerate}
\end{lemma}

With the help of the functions $\gamma$ and $\zeta$ introduced above, we are now ready to parametrize the curve $\mathcal{K}$ defined in \eqref{eq:level_set_K} and to examine some of its properties. For $\alpha, \beta \in [0,1]$, we define the functions
\begin{equation}
\label{eq: u(alpha) v(beta) definition}
u(\alpha) = \log \gamma ( \log \alpha) \quad \text{and} \quad v(\beta) = \log \zeta ( \log \beta).
\end{equation}
These functions provide a parametrization of the curve $\mathcal{K}$, which can be expressed as
\begin{equation*}
\mathcal{K}= \{(\alpha,u(\alpha)):\alpha\in [0,1]\} \cup \{(v(\beta),\beta):\beta\in [0,1]\}.
\end{equation*}
The following lemma outlines some key properties of the functions $u$ and $v$ in \eqref{eq: u(alpha) v(beta) definition}.

\begin{lemma}\label{lem: parametrization level set K}
	The following assertions hold:
	\begin{enumerate}[label=(\roman*)]
	\item $u(1)=v(1)=1$, $u(0)=v(0)=0$, $u(\alpha)> \alpha$ and $v(\beta)> \beta$ for any $\alpha,\beta \in (0,1)$; 
\item $u(\alpha)$ and $v(\beta)$ are infinitely differentiable on $(0,1)$;
\item $u$ has a unique maximizer $\widehat{\alpha}\in (0,1)$, with corresponding value 
\begin{equation*}
 \widehat{\beta}=\max_{\alpha\in [0,1]} u (\alpha) = u(\widehat{\alpha}) >0  . 
\end{equation*}
Moreover, $u$ is strictly increasing on $[0,\widehat{\alpha}]$ and strictly decreasing on $[\widehat{\alpha},1]$;
\item $v$ has a unique maximizer $\widetilde{\beta} \in (0,1)$, with corresponding value 
\begin{equation*}
\widetilde{\alpha}=\max_{\beta\in [0,1]} v (\beta) = v(\widetilde{\beta}) >0.    
\end{equation*}
Moreover, $v$ is strictly increasing on $[0,\widetilde{\beta}]$ and strictly decreasing on $[\widetilde{\beta},1]$.
\end{enumerate}
\end{lemma}
The properties of $u$ and $v$ stated in Lemma~\ref{lem: parametrization level set K} follow directly from their definitions in \eqref{eq: u(alpha) v(beta) definition} and from the properties of $\gamma$ and $\zeta$ established in Lemma~\ref{lem: parametrization level set L}. We therefore omit the proof.

We may now construct the sequence $\{(\alpha_m,\beta_m)\}_{m\in\mathbb{Z}}$ introduced at the beginning of the section, in \eqref{eq:evocation_sequence}. Denoting the inverses of $u$ and $v$ restricted to the intervals $[0,\widehat{\alpha}]$ and $[0,\widetilde{\beta}]$ by
\begin{equation}
\label{eq:def_ustar_vstar}
u^* = \bigl( u\big|{[0,\widehat{\alpha}]} \bigr)^{-1} \quad \text{and} \quad
v^* = \bigl( v\big|{[0,\widetilde{\beta}]} \bigr)^{-1},
\end{equation}
we define the sequence $\{(\alpha_m,\beta_m)\}_{m\in\mathbb{Z}}$ on the curve $\mathcal{K}$ recursively as follows:
\begin{equation}\label{eq: alpha_m beta_m definition}
    \begin{cases}
        \alpha_0=\beta_0 = 1,\\
        \alpha_m = u^*(\beta_{m-1}),\quad \beta_m = v^*(\alpha_m),\quad m\geq 1,\\
        \beta_m = v^*(\alpha_{m+1}),\quad \alpha_m = u^*(\beta_m),\quad m\leq -1.
    \end{cases}
\end{equation}



In the following lemma, we compute the decay of $\alpha_m$ and $\beta_m$ as $m\to\pm\infty$.
\begin{lemma}\label{lem: alpha_m beta_m decay}
	The sequences $\{\alpha_m\}_{m\geq 0}$, $\{\alpha_{-m}\}_{m\geq 0}$, $\{\beta_m\}_{m\geq 0}$, and $\{\beta_{-m}\}_{m\geq 0}$, defined in \eqref{eq: alpha_m beta_m definition}, are strictly decreasing and exhibit exponential decay as $m\to\infty$:
\begin{equation*}
\alpha_{\pm m}, \beta_{\pm m} = O\left( \left( \frac{\widehat{\alpha}}{\widehat{\beta}} \cdot \frac{\widetilde{\beta}}{\widetilde{\alpha}} \right)^{|m|} \right),
\end{equation*}
where $\frac{\widehat{\alpha}}{\widehat{\beta}} \cdot \frac{\widetilde{\beta}}{\widetilde{\alpha}} < 1$, and the constants $\widehat{\alpha}, \widehat{\beta}, \widetilde{\beta}, \widetilde{\alpha}$ are as defined in Lemma~\ref{lem: parametrization level set K}.
\end{lemma}

\begin{proof}
	Since for all $m\geq 1$,
    \begin{equation*}
        \beta_0 = \alpha_0 > \widehat{\alpha} > \alpha_1 \quad\text{and}\quad
        \alpha_m > v^*(\alpha_m) = \beta_m > u^*(\beta_m) = \alpha_{m+1},
    \end{equation*}
then $\{\alpha_m\}_{m\geq 0}$ and $\{\beta_m\}_{m\geq 0}$ are strictly decreasing. An analogous reasoning yields the same result for $\{\alpha_{-m}\}_{m\geq 0}$ and $\{\beta_{-m}\}_{m\geq 0}$.

We move to analyzing the decay of $\alpha_m$ and $\beta_m$ as $m\to\pm\infty$. The idea is to compare them to an alternative sequence $\{ (\alpha'_m, \beta'_m) \}_{m\in\mathbb{Z}}$, which is defined recursively by
\begin{equation*}
    \begin{cases}
        \alpha'_0 = \beta'_0 = 1,\\
        \displaystyle \alpha'_m = \frac{\widehat{\alpha}}{\widehat{\beta}} \beta'_{m-1},\quad \beta'_m = \frac{\widetilde{\beta}}{\widetilde{\alpha}} \alpha'_m, &m\geq 1,\\
        \displaystyle\beta'_m = \frac{\widetilde{\beta}}{\widetilde{\alpha}} \alpha'_{m+1} ,\quad \alpha'_m = \frac{\widehat{\alpha}}{\widehat{\beta}} \beta'_m, &m\leq -1.
    \end{cases}
\end{equation*}

Since the sublevel set $\mathcal{L}^-$ is convex and contains the ray $\{t(1,1):t<0\}$ (Lemma~\ref{lem: level set L}), it must also contain the rays $\{(\widehat{x},\widehat{y})+t(1,1):t<0\}$ and $\{(\widetilde{x},\widetilde{y})+t(1,1):t<0\}$, where $\widehat{x}$ and $ \widehat{y}$ are respectively the maximizer and maximal value of $\gamma$, while $\widetilde{y}$ and $ \widetilde{x}$ are respectively the maximizer and maximal value of $\zeta$, as defined in Lemma~\ref{lem: parametrization level set L}. Consequently, the sublevel set $\mathcal{K}^-=\{(\alpha,\beta)\in[0,\infty)^2:K(\alpha,\beta)<0\}$ contains the segments $\{t(\widehat{\alpha},\widehat{\beta}):t\in(0,1)\}$ and $\{t(\widetilde{\alpha},\widetilde{\beta}):t\in(0,1)\}$. As a result, we have:
\begin{equation*}
	u(\alpha) \geq \frac{\widehat{\beta}}{\widehat{\alpha}} \alpha,\quad \alpha \in [0,\widehat{\alpha}],\quad \text{or equivalently,}\quad 
	u^* (\beta) \leq \frac{\widehat{\alpha}}{\widehat{\beta}} \beta,\quad \beta\in[0,\widehat{\beta}],
\end{equation*}
and
\begin{equation*}
	v(\beta) \geq \frac{\widetilde{\alpha}}{\widetilde{\beta}} \beta,\quad \beta \in [0,\widetilde{\beta}],\quad \text{or equivalently,}\quad 
	v^* (\alpha) \leq \frac{\widetilde{\beta}}{\widetilde{\alpha}} \alpha,\quad \alpha\in[0,\widetilde{\alpha}].
\end{equation*}
These inequalities enable us to compare $\alpha_m$ to $\alpha'_m$, and $\beta_m$ to $\beta'_m$. We first have:
\begin{equation*}
	\alpha_1 = u^* (\beta_0) \leq \frac{\widehat{\alpha}}{\widehat{\beta}}= \alpha'_1\quad
	\text{and}
	\quad
	\beta_1 = v^* (\alpha_1) \leq v^* (\alpha'_1) \leq \frac{\widetilde{\beta}}{\widetilde{\alpha}} \alpha'_1 = \beta'_1.
\end{equation*}
By an induction argument, we then have for any $m\geq 2$:
\begin{align*}
	&\alpha_m = u^*(\beta_{m-1}) \leq u^*(\beta'_{m-1}) \leq \frac{\widehat{\alpha}}{\widehat{\beta}}\beta'_{m-1} = \alpha'_{m},\\
	&\beta_m = v^* (\alpha_m) \leq v^* (\alpha'_m) \leq \frac{\widetilde{\beta}}{\widetilde{\alpha}}\alpha'_m= \beta'_m.
\end{align*}
Similar arguments imply for any $m\leq -1$.
\begin{equation*}
	\alpha_{m} \leq \widehat{\alpha}_{m}\quad
	\text{and}\quad
	\beta_{m} \leq \widehat{\beta}_{m}.
\end{equation*}
Now since the sequence $\{(\alpha'_m,\beta'_m)\}_{m\in\mathbb{Z}}$ can be explicitly expressed by
\begin{equation*}
    \begin{cases}
        \alpha'_0 = \beta'_0 =1,\\
        \displaystyle\frac{\widetilde{\beta}}{\widetilde{\alpha}} \alpha'_m = \beta'_m = \left( \frac{\widehat{\alpha}}{\widehat{\beta}} \cdot\frac{\widetilde{\beta}}{\widetilde{\alpha}} \right)^m , &m\geq 1,\\ \displaystyle
         \alpha'_m = \frac{\widehat{\alpha}}{\widehat{\beta}} \beta'_m = \left( \frac{\widehat{\alpha}}{\widehat{\beta}} \cdot\frac{\widetilde{\beta}}{\widetilde{\alpha}} \right)^{-m} , &m\leq -1,
    \end{cases}
\end{equation*}
it suffices to derive the decay of $\alpha_m$ and $\beta_m$ as $m\to\pm \infty$. The proof is then complete.
\end{proof}

\begin{lemma}\label{lem: G(i,j,z) convergence}
	For any $(i,j)\in\mathbb{N}_0^2\setminus\{(0,0)\}$, the series of functions $G(i,j,z)$ in \eqref{eq:def_function_G} is analytic on the open unit disk, and satisfies Equations~\eqref{eq: recurrence relation F(i,j,z)}--\eqref{eq: terminal condition F(i,j,1)}.
\end{lemma}

\begin{proof}
For clarity, we denote, for any $m \in \mathbb{Z}$,
\begin{equation}
\label{eq:def_v_h_tilde_hat}
    \left\{\begin{array}{rclrcl}
        \widehat{v}_m &=& \sum_{k,\ell} v_{k,\ell} \alpha_m^k \beta_{m}^\ell,\qquad 
	\widetilde{v}_m &=& \sum_{k,\ell} v_{k,\ell} \alpha_m^k \beta_{m-1}^\ell, \smallskip\\
    \widehat{h}_m &=& \sum_{k,\ell} h_{k,\ell} \alpha_m^k \beta_{m}^\ell ,\qquad 
	\widetilde{h}_m &=& \sum_{k,\ell} h_{k,\ell} \alpha_m^k \beta_{m-1}^\ell.
    \end{array}\right.
\end{equation}
Then, the recursive formula \eqref{eq: c_m d_m definition} for the coefficients $c_m(z)$ and $d_m(z)$ becomes:
\begin{equation}
\label{eq:recursive_bis_cd}
	c_0(z) = 1,\quad d_{m+1} (z) = - c_m(z) \frac{1-\widehat{v}_m z}{1-\widetilde{v}_{m+1} z},\quad c_m(z) = -d_m(z) \frac{1-\widetilde{h}_m z}{1-\widehat{h}_m z}.
\end{equation}

We begin by showing that $\sum_{m\geq 1} c_m(z)\alpha_m^i\beta_m^j$ is analytic on the open unit disk.
By direct induction, the coefficients $\{c_m(z)\}_{m\geq 1}$ satisfy the product formula
\begin{equation*}
    c_m(z) = \prod_{k=1}^m \frac{1-\widetilde{h}_kz}{1-\widehat{h}_kz}\frac{1-\widehat{v}_{k-1}z}{1-\widetilde{v}_kz}.
\end{equation*}
If $k$ is large enough, then $\widetilde{h}_k, \widehat{h}_k,\widetilde{v}_k, \widehat{v}_k$ are all in $(0,\varepsilon)$ (recall that $\alpha_m$ and $\beta_m$ decay exponentially as $m \to \pm\infty$, see Lemma~\ref{lem: alpha_m beta_m decay}, and recall our hypothesis \ref{asm: B7}). Hence there exists a constant $C$ such that for all $\vert z\vert \leq 1$ and all $m\geq 1$, 
\begin{equation*}
    c_m(z) \leq C \left(\frac{1+\varepsilon}{1-\varepsilon}\right)^m.
\end{equation*}
Therefore, for any $(i,j)\in\mathbb{N}_0^2\setminus\{(0,0)\}$, the series
$\sum_{m\geq 1} c_m(z)\alpha_m^i\beta_m^j$ is analytic on the disk centered at $0$ with radius $\frac{1 - \varepsilon}{1 + \varepsilon}$.
Since $\varepsilon$ is arbitrary, we conclude that the series is analytic on the open unit disk.

By similar arguments, the series $\sum_{m\leq 0} c_m(z)\alpha_m^i\beta_m^j$ and $\sum_{m\in \mathbb{Z}} d_{m+1}(z)\alpha_{m+1}^i\beta_m^j$ are also analytic on the open unit disk.

We now show that $G(i,j,z)$ satisfies Eqs.~\eqref{eq: recurrence relation F(i,j,z)}--\eqref{eq: terminal condition F(i,j,1)}. Fix $(i,j)\in \mathbb{N}_0^2\setminus\{(0,0)\}$ and $z \in [0,1]$. For any $m \in \mathbb{Z}$, since both $(\alpha_m, \beta_m)$ and $(\alpha_{m+1}, \beta_m)$ belong to $\mathcal{K}$, the terms $\alpha_m^i \beta_m^j$ and $\alpha_{m+1}^i \beta_m^j$ satisfy Eq.~\eqref{eq: recurrence relation F(i,j,z)}. By linearity, it follows that $G(i,j,z)$ satisfies the same equation. Moreover, we set, for any $m \in \mathbb{Z}$
\begin{align*}
	H_m(i,j) &= d_{m}(z) \alpha_{m}^i \beta_{m-1}^j +  c_m(z) \alpha_m^i \beta_m^j ,\\
     V_m(i,j) &= c_m(z) \alpha_m^i \beta_m^j + d_{m+1}(z) \alpha_{m+1}^i \beta_m^j.
\end{align*}
Therefore, the function $G(i,j,z)$ can be written as
\begin{equation*}
	G(i,j,z) = \sum_{m\in\mathbb{Z}} H_m(i,j) = \sum_{m\in\mathbb{Z}} V_m(i,j).
\end{equation*}
It is easily seen that $H_m(i,0)$ and $V_m(0,j)$ satisfy Eq.~\eqref{eq: horizontal condition F(i,j,z)} and Eq.~\eqref{eq: vertical condition F(i,j,z)}, respectively. As a result, $G(i,0) = \sum_{m\in\mathbb{Z}} H_m (i,0)$ and $G(0,j) = \sum_{m\in\mathbb{Z}} V_m (0,j)$ also satisfy Eq.~\eqref{eq: horizontal condition F(i,j,z)} and Eq.~\eqref{eq: vertical condition F(i,j,z)}, respectively. Finally, observe that when $z=1$, all coefficients $c_m(1)$ and $d_m(1)$ vanish except for $c_0(1)\alpha_0^i\beta_0^j=1$, which yields $G(i,j,1) =1$.
\end{proof}

\subsection{Proof of Theorem~\ref{thm: G(i,j,z) generating function}}
In this subsection, we prove that $G(i,j,z)$ is indeed the generating function of the number of contacts $Z_{(i,j)}$, as stated in Theorem~\ref{thm: G(i,j,z) generating function}.

\paragraph{A remark on non-uniqueness of the solutions.}

We should first emphasize that the system of equations \eqref{eq: recurrence relation F(i,j,z)}--\eqref{eq: terminal condition F(i,j,1)} admits infinitely many solutions, among which the series $G(i,j,z)$ in \eqref{eq:def_function_G} is one particular instance. Indeed, alternative solutions can be obtained by altering the initial value $(\alpha_0,\beta_0)$ or the base coefficient $c_0(z)$, while keeping the same recursive structure \eqref{eq: c_m d_m definition} for the remaining $(\alpha_m,\beta_m)$, $c_m(z)$, and $d_m(z)$. More precisely, one can choose any pair $(\alpha_0,\beta_0)$ lying on the curve
\begin{equation*}
\{(\alpha,u(\alpha)):\alpha\in[\widehat{\alpha},1]\} \cup \{(v(\beta),\beta):\beta\in[\widetilde{\beta},1]\},
\end{equation*}
along with any function $c_0(z)$ that is bounded on $[0,1]$. The series $G(i,j,z)$ constructed from these alternative initial values then satisfies Eqs.~\eqref{eq: recurrence relation F(i,j,z)}--\eqref{eq: vertical condition F(i,j,z)}, though it may fail to satisfy the terminal condition \eqref{eq: terminal condition F(i,j,1)}. Conversely, if the initial values are fixed to $\alpha_0 = \beta_0 = 1$, then for any bounded function $c_0(z)$ on $[0,1]$ with $c_0(1) = 1$, the resulting series $G(i,j,z)$ gives another solution to the full system \eqref{eq: recurrence relation F(i,j,z)}--\eqref{eq: terminal condition F(i,j,1)}.

This non-uniqueness leads to difficulties, as the system of equations \eqref{eq: recurrence relation F(i,j,z)}--\eqref{eq: terminal condition F(i,j,1)} alone does not suffice to characterize the solutions, and we will need to invoke certain probabilistic properties of the solution.

\paragraph{Power series expansion of $G(i,j,z)$.}
Thanks to Lemma~\ref{lem: G(i,j,z) convergence}, we may rearrange the terms of $G(i,j,z)$ into a power series in $z$, whose radius of convergence is equal to $1$. For any $m \in \mathbb{Z}$, we write the expansions of $c_m(z)$ and $d_m(z)$ as
\begin{equation*}
	c_m(z) = \sum_{n\geq 0} c_{m,n} z^n \quad\text{and}\quad
	d_m(z) = \sum_{n\geq 0} d_{m,n} z^n.
\end{equation*}
Then $G(i,j,z)$ has the representation $G(i,j,z) = \sum_{n\geq 0} g_n(i,j) z^n$, where the coefficients $g_n(i,j)$ are defined as
\begin{equation}
\label{eq:def_gnij}
	g_n (i,j) =  \sum_{m\in\mathbb{Z}} \bigl( c_{m,n} \alpha_m^i \beta_m^j + d_{m+1,n} \alpha_{m+1}^i \beta_m^j \bigr).
\end{equation}
The series $g_n(i,j)$ in \eqref{eq:def_gnij} converges absolutely for any $n \geq 0$ and $(i,j) \in \mathbb{N}_0^2\setminus\{(0,0)\}$. 

Our goal is to establish the identity $g_n = f_n$ for all $n \geq 0$, where, recall, $f_n(i,j)$ denotes the probability $\mathbf{P}(Z_{(i,j)} = n)$ (see \eqref{eq:def_f_n}). The strategy is to show that both $g_n$ and $f_n$, which are harmonic with respect to the same discrete Laplacian, have identical boundary values and same asymptotic behavior at infinity. We will then apply the maximum principle for discrete harmonic functions to conclude the proof.

Recall that $\alpha_0 = \beta_0 = 1$ and that $\alpha_m, \beta_m \in(0, 1)$ for all $m \neq 0$. It follows that, for any $n \geq 0$, $g_n$ admits the following asymptotic expression as $i+j\to\infty$:
\begin{equation*}
	g_n (i,j) = c_{0,n} + d_{1,n}\alpha_1^i + d_{0,n}\beta_{-1}^j + o(1).
\end{equation*}
Hence only three terms dominate the series as the sum $i + j$ tends to infinity. Fortunately, since the functions $c_0(z)$, $d_1(z)$, and $d_0(z)$ admit particularly simple expressions
\begin{equation*}
	c_0(z) = 1,\quad d_1(z) = - \frac{1-z}{1- \widetilde{v}_1 z},\quad d_0(z) = - \frac{1-z}{1- \widetilde{h}_0z},
\end{equation*}
see \eqref{eq: c_m d_m definition}, we even have explicit expressions for the dominant terms of $g_n$ as $(i+j)\to\infty$:
\begin{equation}\label{eq: g_n asymptotic}
    \begin{cases}\displaystyle
        g_0 (i,j) =1- \alpha_1^i - \beta_{-1}^j + o(1),\\
        g_n(i,j) = \bigl( 1- \widetilde{v}_1 \bigr)    \widetilde{v}_1^{n-1}\alpha_1^i +   \bigl( 1- \widetilde{h}_0 \bigr)    \widetilde{h}_0 ^{n-1}\beta_{-1}^j + o(1),\quad n\geq 1.
    \end{cases}
\end{equation}

\paragraph{Asymptotics of $f_n(i,j)$.}
We now focus on the asymptotic behavior of $f_n(i,j)$ as defined in \eqref{eq:def_f_n}. Our analysis relies on estimating hitting probabilities when the walk starts far from the origin. The following lemma provides the asymptotic probabilities of three complementary events: avoiding the boundary, hitting the vertical axis, and hitting the horizontal axis.

We introduce the functions $f_{[\mathbf{v}]}$ and $f_{[\mathbf{h}]}$, defined for all $(i,j) \in \mathbb{N}_0^2$ by
\begin{equation*}
    f_{[\mathbf{v}]} (i,j) = \mathbf{P} \Bigl( \tau_1^{i,j} < \infty , \Pi_1\bigl(S^{i,j}_{ \tau_1^{i,j} }\bigr) =0\Bigr)\quad \text{and}\quad 
    f_{[\mathbf{h}]} (i,j) = \mathbf{P} \Bigl( \tau_1^{i,j} < \infty , \Pi_2\bigl(S^{i,j}_{ \tau_1^{i,j} }\bigr) =0\Bigr).
\end{equation*}
Here, the subscripts $[\mathbf{v}]$ and $[\mathbf{h}]$ are used to indicate “vertical” and “horizontal,” respectively. The square brackets are included to clearly distinguish these notations from the functions $f_n$.

\begin{lemma}\label{lem: f_0 asymptotic}
With $f_0$ as introduced in \eqref{eq:def_f_n}, and $f_{[\mathbf{v}]}$ and $f_{[\mathbf{h}]}$ as defined above, we have, as $i + j \to \infty$,
	\begin{gather}
	\label{eq: f_0, f hat, f tilde asymptotic behavior}
	f_0(i,j)= 1- \alpha_1^i - \beta_{-1}^j + O\bigl( (\alpha_1\vee \beta_{-1} )^{\tfrac{1}{2}(i+j)} \bigr),\\
	f_{[\mathbf{v}]} (i,j) = \alpha_1^i + O\bigl( (\alpha_1\vee \beta_{-1} )^{\tfrac{1}{2}(i+j)} \bigr),\quad
	f_{[\mathbf{h}]} (i,j)= \beta_{-1}^j +O\bigl( (\alpha_1\vee \beta_{-1} )^{\tfrac{1}{2}(i+j)} \bigr).\notag
\end{gather}
\end{lemma}

Intuitively, Lemma~\ref{lem: f_0 asymptotic} suggests the following: when the walk starts far from both axes, it is most likely to avoid them both; conversely, if it starts far from one axis but close to the other, it rarely hits the distant boundary and behaves, with respect to the nearby axis, much like a one-dimensional random walk.

\begin{proof}
 Recall from \eqref{eq: f_0 upper lower} that we have for all $(i,j)\in\mathbb N_0^2$
	 \begin{equation*}
       \alpha_1^i + \beta_{-1}^j -  \alpha_1^i  \wedge \beta_{-1}^j \leq 1- f_0(i,j) \leq \alpha_1^i + \beta_{-1}^j.
 \end{equation*}
Together with the fact that $\alpha_1^i  \wedge \beta_{-1}^j=O\bigl( (\alpha_1\vee \beta_{-1} )^{\tfrac{1}{2}(i+j)} \bigr)$ as $ i+j\to\infty$, we can deduce the asymptotic behavior of $f_0(i,j)$. Now observe that
\begin{equation}
\label{eq:nobt}
	f_{[\mathbf{v}]}(i,j) + f_{[\mathbf{h}]}(i,j) = 1- f_0(i,j) = \alpha_1^i + \beta_{-1}^j + O\bigl( (\alpha_1\vee \beta_{-1} )^{\tfrac{1}{2}(i+j)} \bigr),
\end{equation}
as $i+j\to\infty$, and based on Lemma~\ref{lem: formula hitting axes},
\begin{equation*}
		f_{[\mathbf{v}]} (i,j) \leq  \mathbf{P} ( \widehat{\tau}^{i,j} < \infty) = \alpha_1^i\quad \text{and}\quad
		f_{[\mathbf{h}]} (i,j) \leq  \mathbf{P} ( \widetilde{\tau}^{i,j} < \infty) = \beta_{-1}^j.
\end{equation*}
Substituting these upper bounds into \eqref{eq:nobt}, we obtain the asymptotic behavior of $f_{[\mathbf{v}]}(i,j)$ and $f_{[\mathbf{h}]}(i,j)$. This concludes the proof.
\end{proof}

We now build on Lemma~\ref{lem: f_0 asymptotic} to analyze the probabilities of subsequent boundary contacts for a random walk starting far from the origin and conditioned on having hit the boundary multiple times before. By the strong Markov property, the walk's behavior under this conditioning depends only on its most recent boundary contact. Consequently, we first introduce the following conditioned probability measures:
\begin{equation*}
        \mathbf{P}_n^{\mathbf{v}} =\mathbf{P}\left(\cdot \left| \tau_n^{i,j} < \infty , \Pi_1 \bigl(S^{i,j}_{\tau_n^{i,j} } \bigr)= 0 \right. \right) 
        \quad\text{and}\quad
        \mathbf{P}_n^{\mathbf{h}} =\mathbf{P}\left(\cdot \left| \tau_n^{i,j} < \infty , \Pi_2 \bigl(S^{i,j}_{\tau_n^{i,j} } \bigr)= 0 \right. \right),
\end{equation*}
for any $(i,j)\in\mathbb N_0^2$ and $n\geq 1$. The following result describes the asymptotic behavior of having one additional boundary contact under these conditioned measures.

\begin{lemma}
\label{lem: probability from axis to axis}
	For any $n\geq 1$, we have as $i+j\to\infty$ 
	\begin{align}
		& \label{eq: asymptotic 2nd contact conditioned 1st contact at vertical line}
        \mathbf{P}_n^{\mathbf{v}} \left(  \tau_{n+1}^{i,j} < \infty \right) = \mathbf{P}_n^{\mathbf{v}} \left(  \tau_{n+1}^{i,j} < \infty, \Pi_1 \bigl(S^{i,j}_{\tau_{n+1}^{i,j} } \bigr)= 0  \right)  + o(1) =   \widetilde{v}_1 + o(1), \\
        & \label{eq: asymptotic 2nd contact conditioned 1st contact at horizontal line}
        \mathbf{P}_n^{\mathbf{h}} \left(  \tau_{n+1}^{i,j} < \infty \right) = \mathbf{P}_n^{\mathbf{v}} \left(  \tau_{n+1}^{i,j} < \infty, \Pi_2 \bigl(S^{i,j}_{\tau_{n+1}^{i,j} } \bigr)= 0  \right)  + o(1) =   \widetilde{h}_0 + o(1).
\end{align}
\end{lemma}
The lemma above conveys a straightforward idea: when a random walk starts far from the origin and reaches one axis, it is unlikely to hit the other axis afterward. Instead, the walk is more likely to either return to the same axis or avoid both axes altogether.

\begin{proof}
We evaluate the probability that the random walk hits the vertical (resp.\ horizontal) axis for the $(n+1)$\textsuperscript{th} time with $n\geq 1$, given that it has previously hit the vertical axis at the $n$\textsuperscript{th} time, as follows:
\begin{align*}
&\mathbf{P}_n^{\mathbf{v}} \left(  \tau_{n+1}^{i,j} < \infty,  \Pi_1 \bigl(S^{i,j}_{\tau_{n+1}^{i,j} } \bigr)= 0 \right) \\
& = \sum_{j'\geq i+j} \mathbf{P}_n^{\mathbf{v}} \left( \Pi_2 \bigl(S^{i,j}_{\tau_n^{i,j} } \bigr)= j'  \right) \mathbf{P}_n^{\mathbf{v}} \left( \tau_{n+1}^{i,j} < \infty,  \Pi_1 \bigl(S^{i,j}_{\tau_{n+1}^{i,j} } \bigr)= 0 \left| \Pi_2 \bigl(S^{i,j}_{\tau_n^{i,j} } \bigr)= j'  \right.\right) \\
&=\sum_{j'\geq i+j} \mathbf{P}_n^{\mathbf{v}} \left( \Pi_2 \bigl(S^{i,j}_{\tau_n^{i,j} } \bigr)= j'  \right) \sum_{k,\ell} v_{k,\ell} f_{[\mathbf{v}]} (k, j'+\ell) \\
&=\sum_{j'\geq i+j} \mathbf{P}_n^{\mathbf{v}} \left( \Pi_2 \bigl(S^{i,j}_{\tau_n^{i,j} } \bigr)= j'  \right) \sum_{k,\ell} v_{k,\ell} \left(\alpha_1^k  + O\bigl( (\alpha_1\vee \beta_{-1} )^{\tfrac{1}{2}(j'+k+\ell)} \bigr) \right) \\
&=  \sum_{k,\ell} v_{k,\ell} \alpha_1^k  + o(1)\\
&= \widetilde{v}_1 + o(1),
\end{align*}
and
\begin{align*}
	&\mathbf{P}_n^{\mathbf{v}}  \left(  \tau_{n+1}^{i,j} < \infty, \Pi_2 \bigl(S^{i,j}_{\tau_{n+1}^{i,j} } \bigr) = 0 \right) \\
    &= \sum_{j'\geq i+j} \mathbf{P}_n^{\mathbf{v}} \left( \Pi_2 \bigl(S^{i,j}_{\tau_n^{i,j} } \bigr)= j'  \right) \mathbf{P}_n^{\mathbf{v}} \left( \tau_{n+1}^{i,j} < \infty, \Pi_2 \bigl(S^{i,j}_{\tau_{n+1}^{i,j} } \bigr)= 0 \left| \Pi_2 \bigl(S^{i,j}_{\tau_n^{i,j} } \bigr)= j' \right. \right) \\
	 & = \sum_{j'\geq i+j} \mathbf{P}_n^{\mathbf{v}} \left( \Pi_2 \bigl(S^{i,j}_{\tau_n^{i,j} } \bigr)= j'  \right) \sum_{k,\ell} v_{k,\ell} f_{[\mathbf{h}]} (k,j'+\ell) \\
	 &=  \sum_{j'\geq i+j} \mathbf{P}_n^{\mathbf{v}} \left( \Pi_2 \bigl(S^{i,j}_{\tau_n^{i,j} } \bigr)= j'  \right) \sum_{k,\ell} v_{k,\ell} \left(\beta_{-1}^{j'+\ell} + O\bigl( (\alpha_1\vee \beta_{-1} )^{\tfrac{1}{2}(j'+k+\ell)} \bigr) \right)\\
	 &= o(1),
\end{align*}
as $i+j\to\infty$. Summing the above two equalities, we deduce \eqref{eq: asymptotic 2nd contact conditioned 1st contact at vertical line}. The proof for \eqref{eq: asymptotic 2nd contact conditioned 1st contact at horizontal line} follows similarly.
\end{proof}

Lemma~\ref{lem: probability from axis to axis} enables us to evaluate the asymptotic probability of $n$ boundary contacts, namely, $f_n (i,j)$, by carefully decomposing the event into a sequence of distinct sub-events corresponding to the first, second, third, and subsequent boundary contacts.

\begin{lemma}
For any $n\geq 1$, we have as $i+j\to\infty$,
\begin{equation}
\label{eq: f_n asymptotic behavior}
    f_n (i,j) =  \bigl( 1- \widetilde{v}_1 \bigr)    \widetilde{v}_1^{n-1}\alpha_1^i +   \bigl( 1- \widetilde{h}_0 \bigr)    \widetilde{h}_0 ^{n-1}\beta_{-1}^j + o(1).
\end{equation}
\end{lemma}

\begin{proof}
By iteratively decomposing the event $\{\tau_1^{i,j}<\infty,\ldots,\tau_n^{i,j}<\infty, \tau_{n+1}^{i,j}=\infty\}$ using the strong Markov property, together with the estimates from Lemma~\ref{lem: probability from axis to axis}, we obtain the following asymptotic evaluation for any $n \geq 1$ as $i + j \to \infty$:
\begin{align*}
	f_n(i,j) &= \mathbf{P} \Bigl(\tau_1^{i,j}<\infty,\ldots,\tau_n^{i,j}<\infty, \tau_{n+1}^{i,j}=\infty,  \Pi_1 \bigl(S^{i,j}_{\tau_1^{i,j} } \bigr)= 0 \Bigr) \\
    &\quad+ \mathbf{P} \Bigl(\tau_1^{i,j}<\infty,\ldots,\tau_n^{i,j}<\infty, \tau_{n+1}^{i,j}=\infty,  \Pi_2 \bigl(S^{i,j}_{\tau_1^{i,j} } \bigr)= 0 \Bigr)  \\
	&= \mathbf{P}_{n}^{\mathbf{v}} \bigl(\tau^{i,j}_{n+1} = \infty \bigr) \mathbf{P}_{n-1}^{\mathbf{v}} \bigl(\tau^{i,j}_{n} < \infty \bigr) \cdots \mathbf{P}_{1}^{\mathbf{v}} \bigl(\tau^{i,j}_{2} < \infty \bigr) \mathbf{P}\Bigl(\tau_1^{i,j}<\infty,  \Pi_1 \bigl(S^{i,j}_{\tau_1^{i,j} } \bigr)= 0 \Bigr) \\
	&\quad +  \mathbf{P}_{n}^{\mathbf{h}} \bigl(\tau^{i,j}_{n+1} = \infty \bigr) \mathbf{P}_{n-1}^{\mathbf{h}} \bigl(\tau^{i,j}_{n} < \infty \bigr)\cdots \mathbf{P}_{1}^{\mathbf{h}} \bigl(\tau^{i,j}_{2} < \infty \bigr) \mathbf{P}\Bigl(\tau_1^{i,j}<\infty,  \Pi_2 \bigl(S^{i,j}_{\tau_1^{i,j} } \bigr)= 0 \Bigr) \\
    &\quad+ o(1)\\
	&= \bigl( 1- \widetilde{v}_1 \bigr)    \widetilde{v}_1^{n-1}\alpha_1^i +   \bigl( 1- \widetilde{h}_0 \bigr)    \widetilde{h}_0 ^{n-1}\beta_{-1}^j + o(1).\qedhere
\end{align*}
\end{proof}

\paragraph{Conclusion of the proof.}
Having established the asymptotic behavior of $f_n(i,j)$, we are now prepared to prove Theorem~\ref{thm: G(i,j,z) generating function}.

\begin{proof}[Proof of Theorem~\ref{thm: G(i,j,z) generating function}]
We prove the identity between $f_n$ and $g_n$ by induction on $n\geq 0$. Starting with the base case $n=0$, note that both $f_0$ and $g_0$ are harmonic on $\mathbb{N}^2$, vanish on the boundary axes (as they satisfy Eqs.~\eqref{eq: recurrence relation f_n(i,j)}--\eqref{eq: vertical condition f_n(i,j)} by construction), and exhibit the same asymptotic behavior as $i+j \to \infty$ (see \eqref{eq: g_n asymptotic} and \eqref{eq: f_0, f hat, f tilde asymptotic behavior}). It follows that their difference $f_0 - g_0$ is also harmonic, vanishes on the boundary, and tends to zero at infinity. By the maximum principle for harmonic functions, we deduce that $f_0 - g_0 = 0$ on $\mathbb{N}_0^2\setminus\{(0,0)\}$, thus establishing the base case.

We now assume that $f_n = g_n$ for some $n \geq 0$, and aim to show that the same equality holds for $n+1$. By construction, both $f_{n+1}$ and $g_{n+1}$ are harmonic on $\mathbb{N}^2$, and they inherit the same boundary values from the equality $f_n = g_n$ established at the previous step. In addition, $f_{n+1}$ and $g_{n+1}$ share the same asymptotic behavior as $i + j \to \infty$ (see \eqref{eq: g_n asymptotic} and \eqref{eq: f_n asymptotic behavior}).
It follows that the difference $f_{n+1} - g_{n+1}$ is harmonic on $\mathbb{N}^2$, vanishes on the boundary, and tends to zero at infinity. Applying the same maximum principle argument as in the base case, we conclude that $f_{n+1} - g_{n+1} = 0$ on $\mathbb{N}_0^2\setminus\{(0,0)\}$. This completes the induction step and establishes that $f_n = g_n$ for all $n \geq 0$.
\end{proof}

\subsection{Asymptotic analytics}
In this subsection, we examine the asymptotic decay of the probability $f_n(i,j)$ as $n\to\infty$. 
\begin{proposition}\label{prop: f_n asymptotic}
    Under Assumptions~\ref{asm: B1}--\ref{asm: B8}, for any $(i,j)\in\mathbb N_0^2\setminus\{(0,0)\}$, the number of contacts $f_n(i,j)$ has the asymptotic behavior as $n\to \infty$:
      \begin{equation}\label{eq: f_n asymptotic in n}
        f_n(i,j) \sim C_{i,j} \left(  \widetilde{v}_1 \vee \widetilde{h}_0 \right)^{n}, 
    \end{equation}
      where
    \begin{equation*}
        C_{i,j} = 
        \begin{cases}
            (1/\widetilde{v}_1-1)S_1(1/\widetilde{v}_1),\quad &\widetilde{v}_1 > \widetilde{h}_0, \\
            (1/\widetilde{v}_1-1)\left[S_1(1/\widetilde{v}_1) + S_0(1/\widetilde{v}_1) \right],\quad &\widetilde{v}_1 = \widetilde{h}_0, \\
            (1/\widetilde{h}_0-1)S_0(1/\widetilde{h}_0),\quad &\widetilde{v}_1 < \widetilde{h}_0,
        \end{cases}  
    \end{equation*}
the functions $S_0(z)$ and $S_1(z)$ being defined by \eqref{eq: S_1(z) S_0(z)}.
\end{proposition}

\begin{proof}
We will show in Lemma~\ref{lem: dominant pole} that the generating function $G(i,j,z)$ has a simple dominant pole.
Proposition~\ref{prop: f_n asymptotic} then follows directly by applying singularity analysis for meromorphic functions; see, for instance, \cite[Thm~IV.10]{FlSe-09}.
The constant $C_{i,j}$ corresponds to the residue of $G(i,j,z)$ at its smallest pole, namely $1/(\widetilde{v}_1 \vee \widetilde{h}_0)$.
\end{proof}


Since $c_m(z)$ and $d_m(z)$ in \eqref{eq: c_m d_m definition} are defined as cumulative products of rational functions, the number of poles in these coefficients increases with $m$. The following result aims to clarify the ordering of these poles.


\begin{lemma}
    We have for any $m\geq 1$,
    \begin{equation*}
        1=\alpha_0=\beta_0 > \alpha_m \vee \beta_{-m} \geq \alpha_m \wedge \beta_{-m} > \alpha_{-m} \vee \beta_{m} \geq \alpha_{-m} \wedge \beta_{m} > \alpha_{m+1} \vee \beta_{-m-1}. 
    \end{equation*}
\end{lemma}
\begin{proof}
    We will only show that \begin{equation*}
        1 > \alpha_1 \vee \beta_{-1} \geq \alpha_1 \wedge \beta_{-1} > \alpha_{-1} \vee \beta_{1} \geq \alpha_{-1} \wedge \beta_{1} > \alpha_{2} \vee \beta_{-2},
    \end{equation*}
    and the inequalities for $m\geq 2$ follow by an induction argument. It is readily seen that $1>\alpha_1,\beta_{-1}$. Now due to the construction and monotonicity of the functions $u^*$ and $v^*$ in \eqref{eq:def_ustar_vstar}, we have:
    \begin{align*}
        &\alpha_1 =u^*(\beta_0) > u^*(\beta_{-1}) = \alpha_{-1}\quad \text{and}\quad
        \beta_{-1} > u^*(\beta_{-1}) = \alpha_{-1},\\
        &\alpha_1 > v^*(\alpha_1) = \beta_1 \quad \text{and}\quad
        \beta_{-1} = v^*(\alpha_0) > v^*(\alpha_1) = \beta_1.
    \end{align*}
    Subsequently, we also have:
    \begin{align*}
        & \alpha_{-1} = u^*(\beta_{-1}) > u^*(\beta_{1}) = \alpha_{2} \quad\text{and}\quad
        \beta_1 > u^*(\beta_1) = \alpha_2 ,\\
        & \alpha_{-1} > v^*(\alpha_{-1}) = \beta_{-2} \quad\text{and}\quad
        \beta_1 = v^*(\alpha_{1}) > v^*(\alpha_{-1}) = \beta_{-2},
    \end{align*}
    which concludes the proof.
\end{proof}
The above lemma yields the precise location of $\widetilde{v}_1$ and $\widetilde{h}_0$ in relation to the values $\widehat{v}_m$, $\widetilde{v}_m$, $\widehat{h}_m$, and $\widetilde{h}_m$ defined in \eqref{eq:def_v_h_tilde_hat}; namely, we have
\begin{equation}
\label{eq:inequalities_vh_0}
    1 > \widetilde{v}_1 \geq \widehat{v}_1\quad \text{and}\quad 1 > \widetilde{h}_0 \geq \widehat{h}_{-1},
\end{equation}
and for $m\geq 1$,
\begin{equation}
\label{eq:inequalities_vh}
\widetilde{v}_1 > \widehat{v}_{m+1},\widehat{v}_{-m},\widetilde{v}_{m+1},\widetilde{v}_{-m}\quad\text{and}\quad 
    \widetilde{h}_0 > \widehat{h}_{m+1}, \widehat{h}_{-m-1}, \widetilde{h}_{m+1}, \widetilde{h}_{-m}.
\end{equation}
It is worth noting that, in general, the relative positions of $\widetilde{v}_1$, $\widetilde{v}_0$, $\widetilde{h}_0$, and $\widetilde{h}_1$ may differ depending on the particular model being studied.




\begin{lemma}\label{lem: dominant pole}
There exists $\varepsilon > 0$ sufficiently small such that, for any $(i,j) \in \mathbb{N}_0^2\setminus\{(0,0)\}$, the function $G(i,j,z)$ is meromorphic in the open disk 
\begin{equation*}
    \mathcal{B}_\varepsilon = \{ z \in \mathbb{C} : |z| < 1/(\widetilde{v}_1 \vee \widetilde{h}_0) + \varepsilon \},
\end{equation*}
and has a unique pole in this domain, located at $1/(\widetilde{v}_1 \vee \widetilde{h}_0)$. This pole is simple.
\end{lemma}

\begin{proof}
Fix $(i,j)\in\mathbb N_0^2\setminus\{(0,0)\}$ for the rest of the proof and, without loss of generality, assume $\widetilde{v}_1 \geq \widetilde{h}_0$. Choose $\varepsilon>0$ so small that $(1/\widetilde{v}_1+\varepsilon)^{-1}$ exceeds all the values $\widehat{v}_m$, $\widehat{h}_m$, $\widetilde{v}_m$, and $\widetilde{h}_m$ cited in \eqref{eq:inequalities_vh_0} and \eqref{eq:inequalities_vh}.

We begin with the case $\widetilde{v}_1 > \widetilde{h}_0$. Since $1/\widetilde{v}_1$ is the unique pole of $d_1(z)=-\frac{1-z}{1-\widetilde{v}_1 z}$, we can use \eqref{eq:def_function_G} together with \eqref{eq: c_m d_m definition} to rewrite $G(i,j,z)$ as
\begin{equation*}
	G(i,j,z) = S_1(z) d_1 (z) + R_1(z),
\end{equation*}
where
\begin{equation*}
	\left\{\begin{array}{rcl}
    S_1(z) &=& \displaystyle \sum_{m\geq 1} \left( \frac{d_m(z)}{d_1 (z)}\alpha_m^i\beta_{m-1}^j +\frac{c_m(z)}{d_1 (z)}\alpha_m^i\beta_{m}^j  \right), \\
     R_1 (z) &=& \displaystyle \sum_{m\leq 0} \left( d_m(z)\alpha_m^i\beta_{m-1}^j + c_m(z)\alpha_m^i\beta_{m}^j  \right).
     \end{array}\right.
\end{equation*}
It suffices to show that $S_1(z)$ and $R_1(z)$ are analytic on $\mathcal{B}_\varepsilon$, and $S_1(1/\widetilde{v}_1)\not=0$. Observe  with \eqref{eq:recursive_bis_cd} that for $m\geq 2$, $d_m(z)/d_1(z)$ can be written as 
\begin{equation*}
    \frac{d_m(z)}{d_1(z)} = \prod_{k=1}^{m-1}
    \frac{1-\widetilde{h}_k z}{1-\widehat{h}_k z}
    \frac{1-\widehat{v}_k z}{1-\widetilde{v}_{k+1}z},
\end{equation*}
of which all the poles are greater than $1/\widetilde{v}_1$. Hence, any $d_m(z)/d_1(z)$ with $m\geq 1$ is analytic on the disk $\mathcal{B}_\varepsilon$. By the same reasoning as in the proof of Lemma~\ref{lem: G(i,j,z) convergence}, we deduce that the series $\sum_{m\geq 1}\frac{d_m(z)}{d_1 (z)}\alpha_m^i\beta_{m-1}^j$ is analytic on $\mathcal{B}_\varepsilon$. Similarly, we conclude that the series $\sum_{m\geq 1}\frac{c_m(z)}{d_1 (z)}\alpha_m^i\beta_{m}^j$, $\sum_{m\leq 0}d_m(z)\alpha_m^i\beta_{m-1}^j$, and $\sum_{m\leq 0}c_m(z)\alpha_m^i\beta_{m}^j$ are all analytic on $\mathcal{B}_\varepsilon$, and hence, so are $S_1(z)$ and $R_1(z)$.

We now show that $S_1(1/\widetilde{v}_1) \not=0$. In the subcase $\widetilde{v}_1 > \widetilde{h}_1$, let us reformulate $S_1(z)$ as
\begin{equation*}
    S_1(z) = \sum_{m\geq 1}  \frac{d_m(z)}{d_1 (z)}\alpha_m^i  \left( \beta_{m-1}^j +\frac{c_m(z)}{d_m(z)}\beta_{m}^j \right) = \sum_{m\geq 1}  \frac{d_m(z)}{d_1 (z)}\alpha_m^i  \left( \beta_{m-1}^j - \frac{1-\widetilde{h}_m z}{1-\widehat{h}_m z}\beta_{m}^j \right).
\end{equation*}
Since
\begin{equation*}
    \frac{d_m(1/\widetilde{v}_1)}{d_1(1/\widetilde{v}_1)} \geq 0,\quad m\geq 2,
\end{equation*}
where the equality happens if $\widetilde{v}_1 = \widehat{v}_1$, and
\begin{equation*}
    \alpha_m^i  \left( \beta_{m-1}^j - \frac{1-\widetilde{h}_m /\widetilde{v}_1}{1-\widehat{h}_m /\widetilde{v}_1}\beta_{m}^j \right) > 0,\quad m\geq 1,
\end{equation*}
then $S_1(1/\widetilde{v}_1) >0$. Now in the subcase $\widetilde{v}_1 \leq \widetilde{h}_1$, we reformulate $S_1(z)$ as
\begin{align*}
    S_1(z)&=\alpha_1^i +\sum_{m\geq 1} \left( \frac{c_m(z)}{d_1(z)}  \alpha_m^i \beta_m^j + \frac{d_{m+1}(z)}{d_1(z)}  \alpha_{m+1}^i \beta_m^j \right) \\
    &= \alpha_1^i +\sum_{m\geq 1} \frac{c_m(z)}{d_1(z)} \beta_m^j \left(  \alpha_m^i  - \frac{1-\widehat{v}_m z}{1-\widetilde{v}_{m+1}z}  \alpha_{m+1}^i  \right).
\end{align*}
Since
\begin{equation*}
    \frac{c_1(\widetilde{v}_1)}{d_1(\widetilde{v}_1)} = -\frac{1-\widetilde{h}_1/\widetilde{v}_1}{1-\widehat{h}_1/\widetilde{v}_1} \geq 0,
\end{equation*}
then
\begin{equation*}
    \frac{c_m(1/\widetilde{v}_1)}{d_1(1/\widetilde{v}_1)} = \frac{c_1(1/\widetilde{v}_1)}{d_1(1/\widetilde{v}_1)} \prod_{k=2}^m \frac{1 - \widehat{v}_{k-1}/\widetilde{v}_1}{1-\widetilde{v}_k /\widetilde{v}_1} \frac{1-\widetilde{h}_k /\widetilde{v}_1 }{1 - \widehat{h}_k /\widetilde{v}_1} \geq 0,\quad m\geq 2.
\end{equation*}
Together with the fact that
\begin{equation*}
    \beta_m^j \left(  \alpha_m^i  - \frac{1-\widehat{v}_m /\widetilde{v}_1}{1-\widetilde{v}_{m+1}/\widetilde{v}_1}  \alpha_{m+1}^i  \right) > 0,\quad m\geq 1,
\end{equation*}
we deduce that $S_1(1/\widetilde{v}_1) >0$, which concludes the case $\widetilde{v}_1 > \widetilde{h}_0$.

We now move to the case $\widetilde{v}_1 = \widetilde{h}_0$. Noting that $1/\widetilde{v}_1$ is the unique simple pole of $d_0(z)=d_1(z)=-\frac{1-z}{1-\widetilde{v}_1z}$, we rewrite $G(i,j,z)$ as
\begin{equation*}
	G(i,j,z)=1 + S_1(z)d_1(z) + S_0(z)d_0(z),
\end{equation*}
where
\begin{equation}\label{eq: S_1(z) S_0(z)}
    \begin{cases}\displaystyle
    S_1(z) = \sum_{m\geq 1} \left( \frac{d_m(z)}{d_1 (z)}\alpha_m^i\beta_{m-1}^j +\frac{c_m(z)}{d_1 (z)}\alpha_m^i\beta_{m}^j  \right),\\ \displaystyle
	S_0(z) = \sum_{m\leq 0 } \left( \frac{d_m(z)}{d_0 (z)}\alpha_m^i\beta_{m-1}^j +\frac{c_{m-1}(z)}{d_0 (z)}\alpha_{m-1}^i\beta_{m-1}^j  \right).
    \end{cases}
\end{equation}
By similar arguments as the preceding case, $S_1(z)$ and $S_0(z)$ are analytic on $\mathcal{B}_\varepsilon$, and $S_1(1/\widetilde{v}_1) > 0$, $S_0(1/\widetilde{h}_0) > 0$. Therefore, the lemma holds in this case as well, which completes the proof.
\end{proof}

We observe that if $1/\widetilde{v}_1$ and $1/\widetilde{h}_0$ are indeed the two smallest simple poles of $G(i,j,z)$ (more precisely, if $\widetilde{v}_1\wedge\widetilde{h}_0 > \widehat{h}_1\vee\widehat{v}_{-1}$), then the asymptotics~\eqref{eq: f_n asymptotic in n} can be rewritten as
\begin{equation*}
f_n(i,j) = (1-\widetilde{v}_1)S_1\left(\frac{1}{\widetilde{v}_1}\right)\widetilde{v}_1^{n-1} + (1-\widetilde{h}_0)S_0\left(\frac{1}{\widetilde{h}_0}\right)\widetilde{h}_0^{n-1} + o\left( ( \widetilde{v}_1 \wedge \widetilde{h}_0 )^{n}\right) ,
\end{equation*}
as $n\to\infty$, for any $(i,j)\in\mathbb{N}_0^2\setminus\{(0,0)\}$. This form closely resembles the asymptotic behavior given in~\eqref{eq: f_n asymptotic behavior} for $f_n(i,j)$ as $i+j\to \infty$. On the other hand, in situations where the drift of the walk (both in the interior and on the boundary of the quadrant) is strongly asymmetric, the ordering of the first few smallest poles may differ significantly across various regimes of transition probabilities. In principle, more refined analyses could yield sharper asymptotic estimates.

\section{Various examples}
\label{sec:examples}

In this section, we explore several examples, particularly those that lie at the intersection of the two models studied in this work (Section~\ref{sec:model-1} and Section~\ref{sec:model-2}). These examples demonstrate the versatility of the compensation approach for both singular and non-singular random walks.


 \subsection{Singular walks with same horizontal and vertical reflections}

We begin with the top-left model in Figure~\ref{fig: model examples}.
\begin{proposition}\label{prop: f_n closed form}
    Assume~\ref{asm: B1}--\ref{asm: B8} and $v_{k,\ell} = h_{k,\ell} =q_{k,\ell} $ for any $k,\ell$. Then, for any $(i,j)\in\mathbb N_0^2$, the generating function $F(i,j,z)$ of $Z_{(i,j)}$ has the expression
    \begin{equation}\label{eq: G(i,j,z) reduced case}
        F(i,j,z) = \sum_{m\in\mathbb{Z}} \left(\frac{1-z}{1-z \sum_{k,\ell}q_{k,\ell}\alpha_m^k \beta_m^\ell} \alpha_m^i\beta_m^j - \frac{1-z}{1- z\sum_{k,\ell}q_{k,\ell}\alpha_{m+1}^k \beta_m^\ell} \alpha_{m+1}^i\beta_m^j\right).
    \end{equation}
    As a consequence, the functions $f_n$ admit the closed-form expressions
    \begin{align*}
        f_0 (i,j) ={}& \sum_{m\in\mathbb{Z}} \left( \alpha_m^i\beta_m^j - \alpha_{m+1}^i\beta_m^j\right), \\
        f_n (i,j) ={}& \sum_{m\in\mathbb{Z}} \left[-\bigl(1 - \sum_{k,\ell}q_{k,\ell}\alpha_m^k \beta_m^\ell\bigr) \bigl(\sum_{k,\ell}q_{k,\ell}\alpha_m^k \beta_m^\ell \bigr)^{n-1}\alpha_m^i\beta_m^j \right.\\
        & \quad\quad\quad\quad\left.+ \bigl(1 - \sum_{k,\ell}q_{k,\ell}\alpha_{m+1}^k \beta_m^\ell\bigr) \bigl(\sum_{k,\ell}q_{k,\ell}\alpha_{m+1}^k \beta_m^\ell \bigr)^{n-1} \alpha_{m+1}^i\beta_m^j \right],\quad n\geq 1.
    \end{align*}
\end{proposition}

Since Assumptions~\ref{asm: A1}--\ref{asm: A6} hold for the model, it is straightforward to verify that the above expressions for $f_n$ satisfy the recurrence relation stated in Theorem~\ref{thm: f_n(x) recurrence relation}.

\begin{proof}
    Since $v_{k,\ell} = h_{k,\ell} =q_{k,\ell} $ for any $k,\ell$, then for any $m\in\mathbb{Z}$,
    \begin{equation*}
        \widehat{v}_m = \widehat{h}_m = \sum_{k,\ell} q_{k,\ell}\alpha_m^k\beta_m^\ell\quad\text{and}\quad
        \widetilde{v}_m = \widetilde{h}_m = \sum_{k,\ell} q_{k,\ell}\alpha_m^k\beta_{m-1}^\ell.
    \end{equation*}
    Accordingly, the coefficients $c_m(z)$ and $d_m(z)$, defined through cumulative products in \eqref{eq: c_m d_m definition}, simplify to the expressions
    \begin{equation*}
        c_m(z) = \frac{1-z}{1-z \sum_{k,\ell}q_{k,\ell}\alpha_m^k \beta_m^\ell}\quad \text{and}\quad
        d_m(z) = - \frac{1-z}{1-z \sum_{k,\ell}q_{k,\ell}\alpha_m^k \beta_{m-1}^\ell},
    \end{equation*}
    which yields the simplified form \eqref{eq: G(i,j,z) reduced case} for the generating function introduced in \eqref{eq:def_function_G}. The closed-form expressions for $f_n$ then follow directly from the Taylor expansion of $F(i,j,z)$ in $z$.
\end{proof}

\begin{figure}
    \centering
    \begin{tikzpicture}[scale=0.8,
    x=5.5mm,
    y=5.5mm,
  ]
\vspace{2cm}
     \draw [fill=white,gray!20,draw opacity=1] (-3,-3) rectangle (6,-1);
      \draw [fill=white,gray!20,draw opacity=1] (-3,-3) rectangle (-1,6);
    \draw[thin]
      \foreach \x in {-3, ..., 6} {
        (\x, -3) -- (\x, 6)
      }
      \foreach \y in {-3, ..., 6} {
        (-3, \y) -- (6, \y)
      }
    ;  

    \begin{scope}[
      semithick,
      ->,
      >={Stealth[]},
    ]
   
   \draw[->,thick,blue] (3,3) -- (4,4);
   \draw[->,thick,blue] (3,3) -- (2,4);
   \draw[->,thick,blue] (3,3) -- (4,2);
      \draw[->,thick,blue] (3,3) -- (5,3);
       \draw[->,thick,blue] (3,3) -- (4,5);
   
   \draw[->,thick,red] (1,-1) -- (1,0);
    \draw[->,thick,red] (1,-1) -- (2,1);
    \draw[->,thick,red] (1,-1) -- (4,0);
    
     \draw[->,thick,red] (-1,1) -- (-1,2);
    \draw[->,thick,red] (-1,1) -- (0,3);
    \draw[->,thick,red] (-1,1) -- (2,2);
    
      \draw (-1, -3.5) -- (-1, 6.5);
      \draw (-3.5,-1) -- (6.5, -1);
      \filldraw[black] (0,0) circle (0pt);
      \node at (0,-1.5) {\textcolor{red}{$v_{k,\ell}=h_{k,\ell}=q_{k,\ell}$}};
      \filldraw[black] (4.9,4.5) circle (0pt) node[]{\textcolor{blue}{$p_{k,\ell}$}};
    \end{scope}
    \end{tikzpicture}\qquad
    \begin{tikzpicture}[scale=0.8,
    x=5.5mm,
    y=5.5mm,
  ]
\vspace{2cm}
     \draw [fill=white,gray!20,draw opacity=1] (-3,-3) rectangle (6,-1);
      \draw [fill=white,gray!20,draw opacity=1] (-3,-3) rectangle (-1,6);
    \draw[thin]
      \foreach \x in {-3, ..., 6} {
        (\x, -3) -- (\x, 6)
      }
      \foreach \y in {-3, ..., 6} {
        (-3, \y) -- (6, \y)
      }
    ;  

    \begin{scope}[
      semithick,
      ->,
      >={Stealth[]},
    ]
   
   \draw[->,thick,blue] (3,3) -- (4,4) node[above right] {$1/3$};
   \draw[->,thick,blue] (3,3) -- (2,4) node[above] {$1/3$};
   \draw[->,thick,blue] (3,3) -- (4,2) node[right] {$1/3$};
   
   \draw[->,thick,red] (1,-1) -- (2,0) node[right] {$1$};
    
     \draw[->,thick,red] (-1,1) -- (0,2) node[right] {$1$};
    
      \draw (-1, -3.5) -- (-1, 6.5);
      \draw (-3.5,-1) -- (6.5, -1);

    \end{scope}
    \end{tikzpicture}
    
       \begin{tikzpicture}[scale=0.8,
    x=5.5mm,
    y=5.5mm,
  ]
\vspace{2cm}
     \draw [fill=white,gray!20,draw opacity=1] (-3,-3) rectangle (6,-1);
      \draw [fill=white,gray!20,draw opacity=1] (-3,-3) rectangle (-1,6);
    \draw[thin]
      \foreach \x in {-3, ..., 6} {
        (\x, -3) -- (\x, 6)
      }
      \foreach \y in {-3, ..., 6} {
        (-3, \y) -- (6, \y)
      }
    ;  

    \begin{scope}[
      semithick,
      ->,
      >={Stealth[]},
    ]
   
   \draw[->,thick,blue] (3,3) -- (4,4) node[above right] {$1/3$};
   \draw[->,thick,blue] (3,3) -- (2,4) node[above] {$1/3$};
   \draw[->,thick,blue] (3,3) -- (4,2) node[right] {$1/3$};
   
   \draw[->,thick,red] (1,-1) -- (2,0) node[right] {$1$};
    
     \draw[->,thick,red] (-1,2) -- (0,2) node[right] {$1$};
    
      \draw (-1, -3.5) -- (-1, 6.5);
      \draw (-3.5,-1) -- (6.5, -1);

    \end{scope}
    \end{tikzpicture}\qquad
        \begin{tikzpicture}[scale=0.8,
    x=5.5mm,
    y=5.5mm,
  ]
\vspace{2cm}
     \draw [fill=white,gray!20,draw opacity=1] (-3,-3) rectangle (6,-1);
      \draw [fill=white,gray!20,draw opacity=1] (-3,-3) rectangle (-1,6);
    \draw[thin]
      \foreach \x in {-3, ..., 6} {
        (\x, -3) -- (\x, 6)
      }
      \foreach \y in {-3, ..., 6} {
        (-3, \y) -- (6, \y)
      }
    ;  

    \begin{scope}[
      semithick,
      ->,
      >={Stealth[]},
    ]
   
   \draw[->,thick,blue] (3,3) -- (3,4) node[above] {$3/8$};
   \draw[->,thick,blue] (3,3) -- (4,3) node[right] {$3/8$};
   \draw[->,thick,blue] (3,3) -- (3,2) node[below] {$1/8$};
   \draw[->,thick,blue] (3,3) -- (2,3) node[left] {$1/8$};
   
   \draw[->,thick,red] (1,-1) -- (2,0) node[right] {$1$};
    
     \draw[->,thick,red] (-1,1) -- (0,2) node[right] {$1$};
    
      \draw (-1, -3.5) -- (-1, 6.5);
      \draw (-3.5,-1) -- (6.5, -1);

    \end{scope}
    \end{tikzpicture}
    \caption{Four examples considered in Section~\ref{sec:examples}}
    \label{fig: model examples}
\end{figure}

\begin{figure}[t]
    \centering
    \includegraphics[width=0.3\linewidth]{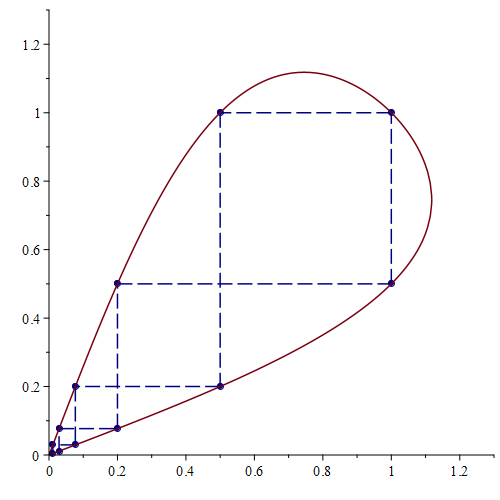}\qquad
    \includegraphics[width=0.3\linewidth]{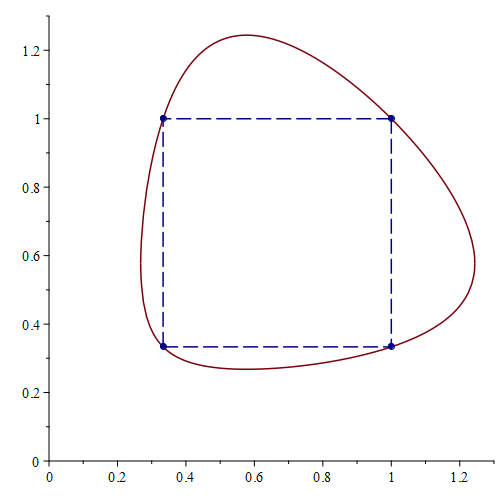}
    \caption{Construction of the sequence $\{(\alpha_m,\beta_m)\}_{m\in\mathbb{Z}}$ in the model $p_{1,1}=p_{1,-1}=p_{-1,1}=\tfrac{1}{3}$ (left) and the model $p_{1,0} = p_{0,1} = \tfrac{3}{8}$, $p_{-1,0} = p_{0,-1} = \tfrac{1}{8}$ (right).}
    \label{fig: alpha_m, beta_m examples}
\end{figure}

We illustrate Proposition~\ref{prop: f_n closed form} with a concrete example. Consider the random walk with weights $p_{-1,1} = p_{1,-1} = p_{1,1} = \frac{1}{3}$ and deterministic reflections given by $v_{1,1} = h_{1,1} = q_{1,1} = 1$; this corresponds to the top-right model in Figure~\ref{fig: model examples}. As shown in \cite{HoRaTa-23}, in this setting the sequence $\{1/\alpha_m, 1/\beta_m\}_{m \in \mathbb{Z}}$ forms a subsequence of the Fibonacci numbers, with the closed-form expression
\begin{equation*}
    (\alpha_m,\beta_m) = \left( \frac{\sqrt{5}}{\rho^{4m-1}+\rho^{-4m+1}}, \frac{\sqrt{5}}{\rho^{4m+1}+\rho^{-4m-1}} \right),\quad \rho=\frac{1+\sqrt{5}}{2},
\end{equation*}
see also Figure~\ref{fig: alpha_m, beta_m examples} (left). More concretely, $\{1/\alpha_m,1/\beta_m\}_{m\in\mathbb{Z}}$ includes the values 
\begin{equation*}
 1,2,5,13,34,89,233,610,1597,\ldots
\end{equation*}
(not necessary in the order indexed by $m$). Applying Proposition~\ref{prop: f_n closed form}, we thus obtain the following expression for the generating function of $Z_{i,j}$:
\begin{multline*}
    F(i,j,z) = \ldots - \frac{1-z}{1-z/65}\cdot \frac{1}{5^i 13^j} + \frac{1-z}{1-z/10}\cdot\frac{1}{5^i 2^j} - \frac{1-z}{1-z/2}\cdot\frac{1}{2^j} \\
     + 1 - \frac{1-z}{1-z/2}\cdot\frac{1}{2^i} + \frac{1-z}{1-z/10}\cdot\frac{1}{2^i 5^j} - \frac{1-z}{1-z/65}\cdot \frac{1}{13^i 5^j} + \ldots,
\end{multline*}
and the expressions for $f_n(i,j)$ as
\begin{align*}
    & f_0(i,j) =  \ldots- \frac{1}{5^i 13^j} + \frac{1}{5^i 2^j} - \frac{1}{2^j} + 1 - \frac{1}{2^i} + \frac{1}{2^i 5^j} - \frac{1}{13^i 5^j} + \ldots ,\\
    & f_n (i,j) = \ldots + \frac{64}{65^{n} 5^i 13^j} - \frac{9}{10^n 5^i 2^j} + \frac{1}{2^n 2^j} + \frac{1}{2^n2^i}  - \frac{9}{10^n 2^i 5^j} + \frac{64}{65^{n} 13^i 5^j} - \ldots,\quad n\geq 1.
\end{align*}

To highlight the constrast in the generating function expressions between models with identical reflections and those with arbitrary reflections, we consider the walk with the weights $ p_{-1,1}=p_{1,-1}=p_{1,1}=\frac{1}{3}$, as in the preceding example, but with reflection probabilities given by $v_{1,0}=h_{1,1}=1$ (bottom-left model in Figure~\ref{fig: model examples}). Then the generating function of $Z_{i,j}$ for this random walk is
\begin{multline*}
    F(i,j,z) = \ldots -  \frac{1-z}{1-\dfrac{z}{2}}\cdot\frac{1-z}{1-\dfrac{z}{5}}\cdot\frac{1-\dfrac{z}{10}}{1-\dfrac{z}{65}}\cdot \frac{1}{5^i 13^j}  + \frac{1-z}{1-\dfrac{z}{2}}\cdot\frac{1-z}{1-\dfrac{z}{5}}\cdot \frac{1}{5^i 2^j}- \frac{1-z}{1-\dfrac{z}{2}}\cdot\frac{1}{2^j}  \\
     +1- \frac{1-z}{1-\dfrac{z}{2}}\cdot\frac{1}{2^i} +\frac{1-z}{1-\dfrac{z}{2}} \cdot\frac{1-\dfrac{z}{2}}{1-\dfrac{z}{10}} \cdot \frac{1}{2^i 5^j} - \frac{1-z}{1-\dfrac{z}{2}} \cdot\frac{1-\dfrac{z}{2}}{1-\dfrac{z}{10}} \cdot\frac{1-\dfrac{z}{2}}{1-\dfrac{z}{13}} \cdot \frac{1}{13^i 5^j} + \ldots
\end{multline*}
In the latter case, the cumulative products do not simplify.

\subsection{Random walks with finite group and identical reflections}

We now adapt the analysis from Section~\ref{sec:model-2} to an example in which the random walk is not singular and exhibits identical horizontal and vertical reflections. Specifically, we consider a biased simple random walk with transition weights given by $p_{1,0} = p_{0,1} = \tfrac{3}{8}$ and $p_{-1,0} = p_{0,-1} = \tfrac{1}{8}$, and where the reflection probabilities satisfy $v_{k,\ell} = h_{k,\ell} = q_{k,\ell}$. Define $\alpha_0=\beta_0=1$ and $\alpha_1=\beta_1=\frac{1}{3}$.

\begin{proposition}\label{prop: f_n finite sum}
    For any $(i,j)\in\mathbb N_0^2$, we have:
    \begin{align*}
        f_0 (i,j) = {}& \alpha_0^i\beta_0^j -\alpha_1^i\beta_0^j + \alpha_1^i\beta_1^j -\alpha_0^i\beta_1^j, \\
        f_n (i,j) = {}&\bigl(1 - \sum_{k,\ell}q_{k,\ell}\alpha_1^k \beta_0^\ell\bigr) \bigl(\sum_{k,\ell}q_{k,\ell}\alpha_1^k \beta_0^\ell \bigr)^{n-1}\alpha_1^i\beta_0^j \\
        &- \bigl(1 - \sum_{k,\ell}q_{k,\ell}\alpha_1^k \beta_1^\ell\bigr) \bigl(\sum_{k,\ell}q_{k,\ell}\alpha_1^k \beta_1^\ell \bigr)^{n-1}\alpha_1^i\beta_1^j \\
        & + \bigl(1 - \sum_{k,\ell}q_{k,\ell}\alpha_0^k \beta_1^\ell\bigr) \bigl(\sum_{k,\ell}q_{k,\ell}\alpha_0^k \beta_1^\ell \bigr)^{n-1}\alpha_0^i\beta_1^j,\quad n\geq 1.
    \end{align*}
The generating function $F(i,j,z)$ of $Z_{i,j}$ is therefore given by~\eqref{eq: G(i,j,z) finite sum}.
\end{proposition}

Before turning to the proof of Proposition~\ref{prop: f_n finite sum}, we briefly explain our motivation for presenting this particular example. This model has the double advantage that the survival probability $f_0$ can be explicitly computed, and that it takes the form of a finite sum of product terms. Such models are extremely rare; they typically arise in the context of small-step walks in the quarter plane, when a certain group naturally associated with the model is finite. See \cite[Chap.~4]{FaIaMa-17} and \cite[Sec.~3]{BMMi-10} for further details. In particular, Table~1 in \cite{BMMi-10} lists models with a finite group, for which the quantity $f_0$ can, in principle, be expressed as a finite sum of exponential terms. For these models, the explicit expression is given in \cite[Cor.~8]{KuRa-11}.

The question of when the stationary distribution of a reflected random walk in the quarter plane can be expressed as a finite sum of exponential terms is addressed, and partially resolved, in \cite{ChBoGo-15,ChBoGo-20}.

Moreover, although this random walk is not a special case of the second model considered in Section~\ref{sec:model-2}, but rather an instance of the first model introduced in Section~\ref{sec:model-1}, we can still apply the compensation approach to construct the generating function of the number of contacts. This demonstrates that the compensation approach extends to a broader class of models.

\begin{proof}[Proof of Proposition~\ref{prop: f_n finite sum}]
Let us begin by introducing the kernel of the model:
\begin{equation*}
    K(\alpha,\beta ) = \frac{3}{8}\alpha^2\beta +  \frac{3}{8}\alpha\beta^2 + \frac{1}{8}\alpha + \frac{1}{8}\beta - \alpha\beta.
\end{equation*}
and its associated level set $\mathcal{K}$ as in \eqref{eq:level_set_K}.
We construct a sequence $\{(\alpha_m,\beta_m)\}_{m\geq 0}$ recursively as follows:
\begin{itemize}
    \item $\alpha_0=\beta_0 =1$;
    \item For any $m\geq 1$, $\alpha_m$ is the root of $K(\cdot,\beta_{m-1})$ such that $\alpha_m\not=\alpha_{m-1}$;
    \item For any $m\geq 1$, $\beta_m$ is the root of $K(\alpha_m,\cdot)$ such that $\beta_m\not=\beta_{m-1}$.
\end{itemize}
Under this construction, we observe that after a finite number of iterations, the pair $(\alpha_m, \beta_m)$ returns to the base value $(\alpha_0, \beta_0)$, so that only four distinct pairs actually appear:
\begin{equation*}
    (\alpha_0,\beta_0) = (1,1),\quad 
    (\alpha_1,\beta_0) = \bigl(\tfrac{1}{3},1\bigr),\quad
    (\alpha_1,\beta_1) = \bigl(\tfrac{1}{3},\tfrac{1}{3}\bigr),\quad
    (\alpha_0,\beta_1) = \bigl(1,\tfrac{1}{3}\bigr),
\end{equation*}
see Figure~\ref{fig: model examples} (right). This is a consequence of the model having a finite group in the sense of \cite{FaIaMa-17,BMMi-10}.

Let us now define the sum $G(i,j,z)$ as
\begin{equation}\label{eq: G(i,j,z) finite sum}
    G(i,j,z) = c_0(z)\alpha_0^i\beta_0^j + d_1(z)\alpha_1^i\beta_0^j + c_1(z)\alpha_1^i\beta_1^j + d_0(z) \alpha_0^i\beta_1^j,
\end{equation}
where 
\begin{align*}
    c_0(z) &= 1,\\
    d_1(z) &= -c_0(z)\frac{1-z\sum_{k,\ell}q_{k,\ell}\alpha_0^k\beta_0^\ell}{1-z\sum_{k,\ell}q_{k,\ell}\alpha_1^k\beta_0^\ell} = -\frac{1-z}{1-z\sum_{k,\ell}q_{k,\ell}\alpha_1^k\beta_0^\ell} ,\\
    c_1(z) &= -d_1(z) \frac{1-z\sum_{k,\ell}q_{k,\ell}\alpha_1^k\beta_0^\ell}{1-z\sum_{k,\ell}q_{k,\ell}\alpha_1^k\beta_1^\ell}=\frac{1-z}{1-z\sum_{k,\ell}q_{k,\ell}\alpha_1^k\beta_1^\ell},\\
    d_0(z) &= -c_1(z) \frac{1-z\sum_{k,\ell}q_{k,\ell}\alpha_1^k\beta_1^\ell}{1-z\sum_{k,\ell}q_{k,\ell}\alpha_0^k\beta_1^\ell}=-\frac{1-z}{1-z\sum_{k,\ell}q_{k,\ell}\alpha_0^k\beta_1^\ell}.
\end{align*}
One can easily check that the sums $\bigl( c_0(z)\alpha_0^i\beta_0^j + d_1(z)\alpha_1^i\beta_0^j \bigr)$ and $ \bigl( c_1(z)\alpha_1^i\beta_1^j + d_0(z) \alpha_0^i\beta_1^j \bigr)$ satisfy Eq.~\eqref{eq: vertical condition F(i,j,z)}, while the sums $\bigl( c_0(z)\alpha_0^i\beta_0^j + d_0(z)\alpha_0^i\beta_1^j \bigr)$ and $ \bigl( d_1(z)\alpha_1^i\beta_0^j + c_1(z) \alpha_1^i\beta_1^j \bigr)$ satisfy Eq.~\eqref{eq: horizontal condition F(i,j,z)}. Together with the fact that each term of $G(i,j,z)$ satisfies Eq.~\eqref{eq: recurrence relation F(i,j,z)} and $G(i,j,1)=1$, it implies that $G(i,j,z)$ satisfies the full system~\eqref{eq: recurrence relation F(i,j,z)}--\eqref{eq: terminal condition F(i,j,1)}.

Specializing \cite[Cor.~8]{KuRa-11} to our current model, one readily verifies that the survival probability is
    \begin{equation*}
        f_0(i,j) = \bigl( 1 - \tfrac{1}{3^i} \bigr) \bigl( 1 - \tfrac{1}{3^j} \bigr),
    \end{equation*}
    which is indeed equal to $\bigl( \alpha_0^i\beta_0^j -\alpha_1^i\beta_0^j + \alpha_1^i\beta_1^j -\alpha_0^i\beta_1^j \bigr)$. The Taylor expansion of $G(i,j,z)$ in the variable $z$ yields the expressions for $f_n$ with $n \geq 1$, which can be readily seen to satisfy the recurrence relation stated in Theorem~\ref{thm: f_n(x) recurrence relation}. This is sufficient to conclude that $G(i,j,z)$ is indeed the generating function of $Z_{i,j}$.
\end{proof} 

As a concrete example, let us consider the same model with reflections at the boundary given by the weights $v_{1,1}=h_{1,1}=q_{1,1}=1$. Applying Proposition~\ref{prop: f_n finite sum}, we find:
\begin{equation*}
\left\{\begin{array}{rcl}
    f_0(i,j) &=& \displaystyle 1 - \frac{1}{3^i} + \frac{1}{3^i 3^j} - \frac{1}{3^j},\smallskip \\
    f_n(i,j) &=&  \displaystyle  \frac{2}{3^n 3^i} - \frac{8}{9^n 3^i 3^j} + \frac{2}{3^n 3^j},\quad n\geq 1.
    \end{array}\right.
\end{equation*}

This example provides evidence that the compensation approach can be effectively applied to non-singular random walks associated with finite groups; or, phrased more in line with the present context, to situations where the sequence of pairs $\{(\alpha_m,\beta_m)\}_{m\geq 0}$ takes only finitely many distinct values, provided that the reflections on the vertical and horizontal axes are identical. In contrast, for non-singular walks with arbitrary boundary reflections, the present computational framework does not produce a valid solution to the system~\eqref{eq: recurrence relation F(i,j,z)}--\eqref{eq: terminal condition F(i,j,1)}. Accordingly, the analysis of this more general case remains an open problem.

\subsection*{Acknowledgments}
KR gratefully acknowledges the Vietnam Institute for Advanced Study in Mathematics in Hanoi and the Industrial University of Ho Chi Minh City for their excellent working conditions during his visits while this paper was being prepared. Both authors thank Nicholas Beaton and Rodolphe Garbit for stimulating and insightful discussions.


\begin{thebibliography}{10}

\bibitem{Ad-91}
I.~J. B.~F. Adan.
\newblock {\em A compensation approach for queueing problems}.
\newblock Eindhoven: Technische Universiteit Eindhoven, 1991.

\bibitem{AdBoRe-01}
I.~J. B.~F. Adan, O.~J. Boxma, and J.~A.~C. Resing.
\newblock Queueing models with multiple waiting lines.
\newblock {\em Queueing Syst.}, 37(1-3):65--98, 2001.

\bibitem{AvVLRa-13}
I.~J. B.~F. Adan, J.~S.~H. Van~Leeuwaarden, and K.~Raschel.
\newblock The compensation approach for walks with small steps in the quarter
  plane.
\newblock {\em Comb. Probab. Comput.}, 22(2):161--183, 2013.

\bibitem{AdWaZi-90}
I.~J. B.~F. Adan, J.~Wessels, and W.~H.~M. Zijm.
\newblock Analysis of the symmetric shortest queue problem.
\newblock {\em Commun. Stat., Stochastic Models}, 6(4):691--713, 1990.

\bibitem{AdWeZi-93}
I.~J. B.~F. Adan, J.~Wessels, and W.~H.~M. Zijm.
\newblock A compensation approach for two-dimensional {Markov} processes.
\newblock {\em Adv. Appl. Probab.}, 25(4):783--817, 1993.

\bibitem{BaFl-02}
C.~Banderier and P.~Flajolet.
\newblock Basic analytic combinatorics of directed lattice paths.
\newblock {\em Theor. Comput. Sci.}, 281(1-2):37--80, 2002.

\bibitem{BaKuWa-24}
C.~Banderier, M.~Kuba, and M.~Wallner.
\newblock Phase transitions of composition schemes: {Mittag}-{Leffler} and
  mixed {Poisson} distributions.
\newblock {\em Ann. Appl. Probab.}, 34(5):4635--4693, 2024.

\bibitem{BaWa-14}
C.~Banderier and M.~Wallner.
\newblock Some reflections on directed lattice paths.
\newblock In {\em Proceedings of the 25th {I}nternational {C}onference on
  {P}robabilistic, {C}ombinatorial and {A}symptotic {M}ethods for the
  {A}nalysis of {A}lgorithms}, volume~BA of {\em Discrete Math. Theor. Comput.
  Sci. Proc.}, pages 25--36. Assoc. Discrete Math. Theor. Comput. Sci., Nancy,
  2014.

\bibitem{BeOwRe-19}
N.~R. Beaton, A.~L. Owczarek, and A.~Rechnitzer.
\newblock Exact solution of some quarter plane walks with interacting
  boundaries.
\newblock {\em Electron. J. Comb.}, 26(3):research paper p3.53, 39, 2019.

\bibitem{BeOwXu-21}
N.~R. Beaton, A.~L. Owczarek, and R.~Xu.
\newblock Quarter-plane lattice paths with interacting boundaries: the
  {Kreweras} and reverse {Kreweras} models.
\newblock In {\em Transcendence in algebra, combinatorics, geometry and number
  theory. TRANS19 -- transient transcendence in Transylvania,
  Bra\textcommabelow{s}ov, Romania, May 13--17, 2019. Revised and extended
  contributions}, pages 163--192. Cham: Springer, 2021.

\bibitem{BoKaVe-21}
A.~Bostan, M.~Kauers, and T.~Verron.
\newblock The generating function of {Kreweras} walks with interacting
  boundaries is not algebraic.
\newblock {\em S{\'e}min. Lothar. Comb.}, 85B:12, 2021.
\newblock Id/No 78.

\bibitem{BMMi-10}
M.~Bousquet-M{\'e}lou and M.~Mishna.
\newblock Walks with small steps in the quarter plane.
\newblock In {\em Algorithmic probability and combinatorics. Papers from the
  AMS special sessions, Chicago, IL, USA, October 5--6, 2007 and Vancouver, BC,
  Canada, October 4--5, 2008}, pages 1--39. Providence, RI: American
  Mathematical Society (AMS), 2010.

\bibitem{BrOwReWh-05}
R.~Brak, A.~L. Owczarek, A.~Rechnitzer, and S.~G. Whittington.
\newblock A direct walk model of a long chain polymer in a slit with attractive
  walls.
\newblock {\em J. Phys. A, Math. Gen.}, 38(20):4309--4325, 2005.

\bibitem{ChBoGo-15}
Y.~Chen, R.~J. Boucherie, and J.~Goseling.
\newblock The invariant measure of random walks in the quarter-plane:
  representation in geometric terms.
\newblock {\em Probab. Eng. Inf. Sci.}, 29(2):233--251, 2015.

\bibitem{ChBoGo-20}
Y.~Chen, R.~J. Boucherie, and J.~Goseling.
\newblock Necessary conditions for the compensation approach for a random walk
  in the quarter-plane.
\newblock {\em Queueing Syst.}, 94(3-4):257--277, 2020.

\bibitem{DrHaRoSi-20}
T.~Dreyfus, C.~Hardouin, J.~Roques, and M.~F. Singer.
\newblock Walks in the quarter plane: genus zero case.
\newblock {\em J. Comb. Theory, Ser. A}, 174:24, 2020.
\newblock Id/No 105251.

\bibitem{ElWa-25}
N.~Elizarov and V.~Wachtel.
\newblock Co-existence of branching populations in random environment.
\newblock Preprint, {arXiv}:2502.18572 [math.{PR}] (2025), 2025.

\bibitem{FaIaMa-17}
G.~Fayolle, R.~Iasnogorodski, and V.~Malyshev.
\newblock {\em Random walks in the quarter plane. {Algebraic} methods, boundary
  value problems, applications to queueing systems and analytic combinatorics},
  volume~40 of {\em Probab. Theory Stoch. Model.}
\newblock Cham: Springer, 2nd edition, previously published with the subtitle
  {Algebraic} methods, boundary value problems and applications edition, 2017.

\bibitem{Fe-68}
W.~Feller.
\newblock An introduction to probability theory and its applications. {I}.
\newblock New {York}-{London}-{Sydney}: {John} {Wiley} and {Sons}, {Inc}.
  {XVIII}, 509 p. (1968)., 1968.

\bibitem{FlSe-09}
P.~Flajolet and R.~Sedgewick.
\newblock {\em Analytic Combinatorics}.
\newblock Cambridge University Press, 2009.

\bibitem{GaRa-14}
R.~Garbit and K.~Raschel.
\newblock On the exit time from a cone for {Brownian} motion with drift.
\newblock {\em Electron. J. Probab.}, 19:27, 2014.
\newblock Id/No 63.

\bibitem{HoRaTa-23}
V.~H. Hoang, K.~Raschel, and P.~Tarrago.
\newblock Harmonic functions for singular quadrant walks.
\newblock {\em Indag. Math., New Ser.}, 34(5):936--972, 2023.

\bibitem{KuRa-11}
I.~Kurkova and K.~Raschel.
\newblock Random walks in {{\((\mathbb Z_+)^2\)}} with non-zero drift absorbed
  at the axes.
\newblock {\em Bull. Soc. Math. Fr.}, 139(3):341--387, 2011.

\bibitem{OwReWo-12}
A.~L. Owczarek, A.~Rechnitzer, and T.~Wong.
\newblock Exact solution of two friendly walks above a sticky wall with single
  and double interactions.
\newblock {\em J. Phys. A, Math. Theor.}, 45(42):23, 2012.
\newblock Id/No 425003.

\bibitem{TaOwRe-14}
R.~Tabbara, A.~L. Owczarek, and A.~Rechnitzer.
\newblock An exact solution of two friendly interacting directed walks near a
  sticky wall.
\newblock {\em J. Phys. A, Math. Theor.}, 47(1):34, 2014.
\newblock Id/No 015202.

\bibitem{TaOwRe-16}
R.~Tabbara, A.~L. Owczarek, and A.~Rechnitzer.
\newblock An exact solution of three interacting friendly walks in the bulk.
\newblock {\em J. Phys. A, Math. Theor.}, 49(15):27, 2016.
\newblock Id/No 154004.

\bibitem{Wa-20}
M.~Wallner.
\newblock A half-normal distribution scheme for generating functions.
\newblock {\em Eur. J. Comb.}, 87:20, 2020.
\newblock Id/No 103138.

\end{thebibliography}
\end{document}